\documentclass[letterpaper,10pt,english]{article}
\usepackage{babel}
\usepackage{amsmath,amsthm,amssymb,amsfonts,indentfirst}

\newtheorem{theo}{Theorem}
\newtheorem{cor}{Corollary}
\newtheorem{lem}{Lemma}
\newtheorem{prop}{Proposition}

\theoremstyle{remark}

\newtheorem{dfn}{\bf Definition}

\numberwithin{equation}{section} \numberwithin{theo}{section}
\numberwithin{cor}{section} \numberwithin{lem}{section}
\numberwithin{prop}{section} \numberwithin{claim}{section}
\numberwithin{obs}{section} \numberwithin{dfn}{section}

\newcommand\R{\text{I\!R}}
\newcommand\N{\text{I\!N}}

\newcommand\de{\delta}
\newcommand\be{\beta}
\newcommand\al{\alpha}

\newcommand{\fr}{\partial}
\newcommand{\equ}{\eqref}
\newcommand{\grad}{\nabla}
\newcommand{\ml}{\mathcal}

\newcommand{\DD}{\ml{D}}
\newcommand{\sm}{\setminus}
\newcommand{\la}{\lambda}
\newcommand{\st}{such that }
\newcommand{\dem}{\bf Proof:}

\newcommand\lap{\Delta}
\newcommand\lab{\Delta_g}

\newcommand\ti{\tilde}

\newcommand{\lf}{\left}
\newcommand{\rg}{\right}

\newcommand\ds{\displaystyle}

\newcommand{\bebs}{\begin{equation*}\begin{split}}
\newcommand{\ee}{\end{equation*}}
\newcommand{\esp}{\end{split}}

\DeclareMathAlphabet{\mathpzc}{OT1}{pcz}{m}{it}

\begin{document}
\title{Bubbling solutions for mean field equations with variable intensities on compact Riemann surfaces}
\author{Pablo Figueroa \thanks{Instituto de Ciencias Físicas y Matemáticas, Facultad de  Ciencias, Universidad Austral de Chile, Campus Isla Teja, Valdivia, Chile. E-mail: pablo.figueroa@uach.cl. Author partially supported by grant Fondecyt Regular Nº1201884, Chile. }}
\date{\today}
\maketitle

\begin{abstract} 
\noindent For an asymmetric sinh-Poisson problem arising as a mean field equation of equilibrium turbulence vortices with variable intensities of interest in hydrodynamic turbulence, we address the existence of bubbling solutions on compact Riemann surfaces. By using a Lyapunov-Schmidt reduction, we find sufficient conditions under which there exist bubbling solutions blowing up  at $m$ different points of $S$: positively at $m_1$ points and negatively at $m-m_1$ points with $m\ge 1$ and $m_1\in\{0,1,...,m\}$. 
Several examples in different situations illustrate our results in the sphere $\mathbb S^2$ and flat two-torus $\mathbb T$ including non negative potentials with zero set non empty. \\[0.1cm]
\end{abstract}

\\[0.1cm]


\section{Introduction}

\noindent Let $(S,g)$ be a compact Riemann surface and consider the problem
\begin{equation}\label{mfewt}
-\lab u=\la_1\lf({V_1(x)e^{u}\over\int_{S}V_1 e^{u}dv_g}-
\frac{1}{|S|}\rg)-\la_2\tau \lf({V_2(x)e^{-\tau u}\over\int_{S}V_2 e^{-\tau u}dv_g}-
\frac{1}{|S|}\rg),
\end{equation}
where $\lambda_1,\la_2\ge0$, $\tau>0$, $V_1$ and $V_2$ are smooth nonnegative potentials in $S$ and $|S|$ is the area of $S$. Here, $\lab$ is the Laplace-Beltrami operator and $dv_g$ is the area element in $(S,g)$. This equation 
have attracted a lot of attention in recent years due to its relevance in the statistical mechanics description of 2D-turbulence, as initiated by Onsager \cite{o}. Precisely, in this context, under a \emph{deterministic} assumption on the distribution of the vortex circulations, Sawada, Suzuki \cite{ss} derive the following equation:
\begin{equation}\label{p12}
\begin{array}{ll}
\ds -\Delta_g u=\lambda \int\limits_{[-1,1]}\alpha \bigg({e^{\alpha u}\over \int_S e^{\alpha u}dv_g} - {1\over |S|}\bigg) d \mathcal P(\alpha)& \hbox{in}\ S
\end{array}
\end{equation}
where $u$ is the stream function of a turbulent Euler flow, $\lambda>0$ is a physical constant related to the inverse temperature and  $\mathcal P$ is a Borel probability measure in $[-1,1]$ describing the point-vortex intensities distribution.

\medskip \noindent Equation \eqref{p12} includes several well-known problems depending on a suitable choice of $\ml P$. For instance, if $\mathcal P=\delta_1$ is concentrated at $1$, then \eqref{p12} is related to the classical mean field equation
\begin{equation}\label{mfe}
-\Delta_g u=\lambda\lf(\frac{Ve^{u}}{\int_S Ve^{u}\,dv_g} -
\frac{1}{|S|}\rg)\quad\text{in}\quad S,
\end{equation}
where $V$ is a smooth nonnegative function on $S$. The latter equation has been studied in several contexts such as conformal
geometry \cite{ChY,ChGY,KW}, statistical mechanics \cite{CLMP1,CLMP2,ChKi,K} and the relativistic Chern-Simons-Higgs model when $S$ is a flat two-torus \cite{NT,T,Tbook}. Notice that solutions of \eqref{mfe} are critical points of the functional
\begin{equation*}	
J_\la(u)={1\over 2}\int_S |\nabla u|^2_g\,dv_g-\la \log\lf(\int_S V e^{u}\, dv_g \rg),\qquad u\in \bar H,
\end{equation*}
where $\bar H=\{u \in H^1(S): \int_S u dv_g=0\}$. Minimizers of $J_\la$ for $\la<8\pi$ can be found by using Moser-Trudinger's inequality. The situation in the supercritical regime $\la\ge 8\pi$ becomes subtler and the existence of solutions could depend on the topology and the geometry of the surface $S$ (or the domain). A degree argument has been proved in \cite{CL0,CL} by Chen and Lin, completing a program initiated by Li \cite{Li}, and has received a variational counterpart in \cite{Dja,Mal} by means of  improved forms of the Moser-Trudinger inequality. 

\medskip \noindent Equation \eqref{mfewt} is also related to \eqref{p12} when $\mathcal P=\sigma \delta_{1}+(1-\sigma)\delta_{-\tau }$ with $\tau \in[-1,1]$  and $\sigma\in[0,1]$. Furthermore, \eqref{mfewt} is the Euler-Lagrange equation of the functional
\begin{equation}\label{energy}
J_{\lambda_1,\la_2}(u)={1\over2}\int_S|\grad u|_g^2\, dv_g -
\lambda_1\log\lf(\int_S V_1 e^{u} dv_g\rg)- \lambda_2\log\lf(\int_S V_2
e^{-\tau u} dv_g\rg),\:\: u\in \bar H.
\end{equation}
If $\tau =1$ and $V_1=V_2\equiv 1$ problem \eqref{mfewt} reduces to mean field equation of the equilibrium turbulence, see \cite{bjmr,J0,JWY,OhSu,R} or its related sinh-Poisson version, see \cite{BaPi,BaPiWe,GP,JWY2,JWYZ}, which have received a considerable interest in recent years. Precisely, in \cite{OhSu} a Trudinger-Moser type inequality was proved: if
$\lambda_1,\lambda_2\in[0,8\pi)$, which can be called the subcritical case, then solutions to \eqref{mfewt} are the minimizers of $J_{\lambda_1,\lambda_2}$, since this functional is coercive; but if $\lambda_1,\lambda_2\in[0,8\pi]$ and either $\lambda_1=8\pi$ or $\lambda_2=8\pi$
then the functional $J_{\lambda_1,\lambda_2}$ still has lower bound but it is not coercive.  A minimization technique is no longer possible if $\la_i> 8\pi $ for some $i=1,2$ since $J_{\la_1,\la_2}$ becomes unbounded from below. In general, one needs to apply variational methods to obtain the existence of critical points (generally of saddle type) for $J_{\la_1,\la_2}$. Several results in the supercritical case can be found in \cite{R,Zhou0,Zhou}. A quantization property was derived in \cite{JWY2} for a blow-up sequence $\{u_n\}_n$ to \eqref{mfewt} with $\tau=1$, one has
\begin{equation}\label{m12}
m_k(p) = \lim_{r\to0}\lim_{n\to+\infty}\frac{\la_{k,n}\int_{B_r(p)}V_k e^{(-1)^{k-1} u_n}\,dv_g}{\int_S V_k e^{(-1)^{k-1} u_n}\,dv_g}\in 8\pi \N,\quad k=1,2,
\end{equation}
extending the corresponding ones for \eqref{mfe} in \cite{LSh} and for \eqref{mfewt} with $\tau=1$ and $V_1= V_2\equiv 1$ in \cite{JWYZ}.

\medskip \noindent Concerning the version of problem \eqref{mfewt} on bounded domains Pistoia and Ricciardi built in \cite{pr1} sequences of blowing-up solutions when $\tau  >0$ and $\lambda_1,\lambda_2\tau^2$ are close to $8\pi$, while in \cite{pr2} the same authors built an arbitrary large number of  sign-changing blowing-up solutions when $\tau  >0$ and $\lambda_1,\lambda_2\tau^2$ are close to suitable (not necessarily integer) multiples of $8\pi.$  Ricciardi and Takahashi in \cite{rt}  provided a complete blow-up picture for solution sequences of \eqref{mfewt}   and successively  in \cite{rtzz} Ricciardi et al. constructed min-max solutions  when $\lambda_1 \to 8\pi^+$ and   $\lambda_2 \to 0$ on a multiply connected domain (in this case the nonlinearity $e^{-\tau  u}$ 
  may  be treated as a lower-order term with respect to the main term $e^u$).    

\medskip \noindent In a compact Riemann surface $S$, a blow-up analysis in subcritical case $\la_1<8\pi$ and $\la_2<\frac{8\pi}{\tau^2}$, and supercritical case $\la_1<16\pi$ and $\la_2<\frac{16\pi}{\tau^2}$, characterizing the blow-up masses $m_k(p)$, $k=1,2$, defined similarly as in \eqref{m12}, has been obtained in \cite{j2}, when $0<\tau<1$. Furthermore, some existence results are deduced. The authors in \cite{rz} obtain the minimal blow-up masses and proved an existence result which generalize the one obtained in \cite{R} for $\tau=1$.

\medskip \noindent To the extend our knowledge, there are by now just few results concerning the existence of bubbling solutions to \eqref{mfewt} and its variants in different framework. For instance bubbling solutions have been constructed for a sinh-Poisson equation ($\tau=1$) on bounded domains in \cite{BaPi,BaPiWe}  with Dirichlet boundary condition and recently in \cite{FIT} with Robin boundary condition. Furthermore, recently in \cite{EFP} and \cite{F2}, the authors have constructed blowing-up solutions on pierced domains with Dirichlet boundary condition for any $\tau>0$. See also \cite{pr1,pr2} for generalizations to $\tau>0$ of results obtained in \cite{BaPi,GP} for $\tau=1$, respectively. The construction of sign-changing bubble tower solutions for sinh-Poisson type equations on pierced domains have been addressed in \cite{F3}.

\medskip \noindent By following some ideas presented in \cite{BaPi,EF}, we are interested in to construct bubbling solutions $u_{\la_1,\la_2}$ to \eqref{mfewt} with $m_1$ positive bubbles and $m_2$ negative bubbles suitable centered at $m=m_1+m_2$ different points of $S$ as both $\la_1\to
8\pi m_1$ and $\la_2\tau^2 \to8\pi m_2$, with $m_1\in\{0,\dots,m\}$. To this aim, introduce the Green function $G(x,p)$ with pole at $p \in S$ as the solution of
\begin{equation} \label{green}
\left\{ \begin{array}{ll} -\Delta_g G(\cdot,p)= \delta_{p}-\frac{1}{|S|} &\text{in $S$}\\
\int_S G(x,p)dv_g=0 &
\end{array} \right.
\end{equation}
where $\delta_p$ denote a Dirac mass in $p\in S$. Define for $\xi=(\xi_1,\dots,\xi_m)\in \tilde S^{m}\sm\Delta$ the functional
\begin{equation}\label{fim}
\begin{split}
\varphi_m^* (\xi) = &\ \frac{1}{4\pi}\sum_{j=1}^{m_1} \log V_1(\xi_j ) + \frac{1}{4\pi \tau^2}\sum_{j=m_1+1}^{m} \log V_2(\xi_j ) + \sum_{j=1}^{m_1} H(\xi_j,\xi_j) + {1\over\tau^2} \sum_{j=m_1+1}^{m} H(\xi_j,\xi_j)\\
& +\sum_{j=1}^{m_1}\sum_{i=1\atop i\not= j}^{m_1} G(\xi_i,\xi_j) - {2\over\tau} \sum_{j=1}^{m_1}\sum_{i=m_1+1}^{m} G(\xi_i,\xi_j) + {1\over\tau^2} \sum_{j=m_1+1}^{m}\sum_{i=m_1+1\atop i\not= j}^{m} G(\xi_i,\xi_j),
\end{split}
\end{equation}
where $H(x,\xi)$ is the regular part of $G(x,\xi)$, $\tilde S=\{V_1,V_2>0\}$ and $\Delta=\{\xi \in S^m:\,\xi_i=\xi_j \hbox{ for }i\not=j\}$ is the diagonal set in $S^m$ with $m=m_1+m_2$. Setting for $j\in \ml J_1:=\{1,\dots,m_1\}$ 
\begin{equation}\label{ro1}
\rho_j(x):=V_1(x)\exp\bigg(8\pi H(x,\xi_j)+ 8\pi \sum\limits_{i=1\atop i\ne j}^{m_1}
G(x,\xi_i)-{8\pi \over\tau}\sum_{i=m_1+1}^mG(x,\xi_i)\bigg),
\end{equation}
and for $j\in \ml J_2:=\{m_1+1,\dots,m\}$ 
\begin{equation}\label{ro2}
\rho_j(x):=V_2(x)\exp\bigg(8\pi H(x,\xi_j)- 8\pi \tau \sum_{i=1}^{m_1}
G(x,\xi_i)+ 8\pi \sum_{i=m_1+1\atop i\ne j}^mG(x,\xi_i)\bigg),
\end{equation}
both 	for $\xi \in S^m \setminus \Delta$ we introduce the notation
\begin{equation}\label{vk}
A_k^*(\xi)=4\pi\sum_{j\in \ml J_k}\left[\Delta_g
\rho_j(\xi_j)-2K(\xi_j)\rho_j(\xi_j)\right],\quad k=1,2
\end{equation}
where $K$ is the Gaussian curvature of $(S,g)$. The sign of $A_k^*$, $k=1,2$ allows us to obtain a first existence result of bubbling solutions and several consequences, see Theorem \ref{main1} and section \ref{secex}. Unfortunately, there are cases where the sign of $A_k^*(\xi)$ either $k=1$ or $k=2$ or both is not available. For instance the case $S=\mathbb T$, $V_1=V_2\equiv1$, $m_1=m_2=1$ and $\tau=1$. See also \cite{EF} for several examples in case $\la_2=0$, namely, $m_2=0$, that could be extended here. Following ideas presented in \cite{EF}, in all these situations, a more refined analysis is necessary. To this aim, introduce the quantities for $k=1,2$
\begin{eqnarray}\label{Bk}
B_k^*(\xi)&\hspace{-0.1cm}=&\hspace{-0.1cm} -2\pi \sum_{j\in \ml J_k} [\Delta_g  \rho_j(\xi_j) -2 K(\xi_j) \rho_j(\xi_j)] \log \rho_j(\xi_j) -\frac{A_k^*(\xi)}{2}\\
&\hspace{-0.1cm}&\hspace{-0.1cm}+ \lim_{r\to0} \bigg[8 \int_{S \setminus \cup_{j\in \ml J_k} B_r(\xi_j)}   V_1e^{8\pi(-\tau)^{k-1} \sum\limits_{j=1}^{m_1} G(x,\xi_j) + 8\pi (-\tau)^{k-2 }\sum\limits_{l=m_1+1}^m G(x,\xi_l) } dv_g-\frac{8\pi}{r^2} \sum\limits_{j\in \ml J_k}  \rho_j(\xi_j)\nonumber\\
&&\qquad\qquad-A_k^*(\xi) \log \frac{1}{r}\bigg]\nonumber
\end{eqnarray}
where $B_r(\xi)$ denotes the pre-image of $B_r(0)$ through the isothermal coordinate system at $\xi$. These types of quantities has been first used and derived by Chang, Chen and Lin in \cite{ChChL} in the study of the mean field equation on bounded domains with Dirichlet boundary condition; for the case of the torus see \cite{CLW}. Moreover, the constant $B_k^*(\xi)$ has also been used in the construction of non-topological condensates for the relativistic abelian Chern-Simons-Higgs model as the Chern-Simons parameter tends to zero, see \cite{DEFM,EF,LinYan}. Our main result states as follows.
\begin{theo} \label{main2}
Let $\mathcal{D} \subset \subset \tilde S^m \setminus \Delta$ be a stable critical set of $\varphi_m^*$. Assume that
\begin{equation} \label{cond0}
\hbox{either }A_1^*(\xi) >0 \:(<0 \hbox{ resp.)} \qquad \hbox{or} \qquad A_1^*(\xi)=0, \:B_1^*(\xi)>0 \: (<0 \hbox{ resp.)}
\end{equation}
and
\begin{equation} \label{cond1}
\hbox{either }A_2^*(\xi) >0 \:(<0 \hbox{ resp.)} \qquad \hbox{or} \qquad A_2^*(\xi)=0, \:B_2^*(\xi)>0 \: (<0 \hbox{ resp.)}
\end{equation}
do hold in a closed neighborhood $U$ of $\mathcal{D}$ in $\tilde S^m \setminus \Delta$. Then, for all $\lambda_1$ in a small right (left resp.) neighborhood of $8 \pi m_1$ and $\lambda_2\tau^2$ in a small right (left resp.) neighborhood of $8 \pi m_2$ there is a solution $u_{\lambda_1,\la_2}$ of \eqref{mfewt} which concentrate (along sub-sequences) at $m$ points, positively at $q_1,\dots, q_{m_1}$ and negatively at $q_{m_1+1},\dots,q_m$, in the sense
\begin{equation}\label{conc}
\frac{\la_1V_1e^{u_{\la_1,\la_2}}}{\int_S V_1e^{u_{\la_1,\la_2}} dv_g}\rightharpoonup 8\pi \sum_{j=1}^{m_1}\de_{q_j}\quad\text{ and }\quad \frac{\la_2\tau^{2}V_2e^{ -\tau u_{\la_1,\la_2}}}{\int_S V_2e^{-\tau u_{\la_1,\la_2}} dv_g}\rightharpoonup 8\pi \sum_{j=m_1+1}^m\de_{q_j}
\end{equation}
as simultaneously $\la_1 \to 8\pi m_1$ and $\la_2\tau^{2}\to 8\pi m_2$ for some $q \in \mathcal{D}$.
\end{theo}
\noindent Notice that along with \eqref{conc} there hold $(-\tau)^{k-1}u_{\la_1,\la_2} - \log\int_S V_ke^{(-\tau)^{k-1}u_{\la_1,\la_2} }\to -\infty$ in $C_{\text{loc}}(S\sm\{q_1,\dots,q_m\})$ and
$$\sup_{\ml O_j}\bigg((-\tau)^{k-1}u_{\la_1,\la_2} - \log\int_S V_ke^{(-\tau)^{k-1}u_{\la_1,\la_2} }\bigg) \to +\infty$$
as simultaneously $\la_1 \to 8\pi m_1$ and $\la_2\tau^{2}\to 8\pi m_2$, for any neighborhood $\ml O_j$ of $q_j$ in $S$ with $k=1$ for $j=1,\dots,m_1$ and $k=2$ for $j=m_1+1,\dots,m$. Hence, we get that $u_{\la_1,\la_2}$ concentrates positively at $q_1,\dots,q_{m_1}$ and negatively at $q_{m_1+1},\dots,q_m$ as simultaneously $\la_1 \to 8\pi m_1$ and $\la_2\tau^{2}\to 8\pi m_2$. As in \cite{EF}, the notion of stability we are using here is the one introduced in \cite{Li0}, see Definition \ref{stable} below. Conditions (\ref{cond0})-\eqref{cond1} on a neighborhood of $\DD$ are required to deal a with stable critical set $\DD$ in the sense below. Arguing as in Remark 4.5 in \cite{EF}, the same conclusion of Theorem \ref{main2} follows  under the validity of conditions (\ref{cond0})-\eqref{cond1} just on $\DD=\{\xi_0\}$, where $\xi_0$ is a non-degenerate local minimum/maximum point of $\varphi_m^*$. Similarly, Theorem \ref{main2} is also valid in the special case $|A_k^*(\xi)|=O(|\nabla \varphi_m^*(\xi)|_g)$, $k=1,2$ in a neighborhood of $\DD$ and $B_k^*(\xi)>0$ in $\DD$.

\medskip Now, we can address the case $S=\mathbb T$, $V_1=V_2\equiv1$, $m_1=m_2=1$ and $\tau=1$. When $\mathbb T$ is a rectangle, the constants like $B_k^*(\xi)$, $k=1,2$, has been used by Chen, Lin nd Wang \cite{CLW} in the computation of the Leray-Schauder degree. Due to $H(x,x)$ is constant in $\mathbb T$, we deduce that $\varphi_2^*(\xi)=-2G(\xi_1,\xi_2)+\text{const.}$. Also, it is known that the Green's function satisfies $G(\xi_1,\xi_2)=G(\xi_1-\xi_2,0)$ and the function $G(\cdot,0)$ has exactly three non-degenerate critical points $q_1$, $q_2$ (saddle points) and $q_3$ (minimum point). According to \eqref{Bk} we have that for $i,k\in\{1,2\}$
$$B_k^*(\xi)=\lim_{r\to0}\lf[8\int_{\mathbb T\sm B_r(\xi_k)} e^{8\pi G(x,\xi_k)- 8\pi G(x,\xi_i)} - {8\pi\over r^2}e^{8\pi H(\xi_k,\xi_k)-8\pi G(\xi_i,\xi_k)}\rg], \ \ i\ne k.$$
Assuming that $\mathbb T=-\mathbb T$ it follows that $B_1^*(\xi)=B_2^*(\xi)$, $\xi=(\xi_1,\xi_2)$, since $G(z,0)=G(-z,0)$. Furthermore, it is known that $B_1^*(\xi)>0$ when either $\xi_1-\xi_2=q_1$ or $\xi_1-\xi_2=q_2$ and $B_1^*(\xi)<0$ when either $\xi_1-\xi_2=q_3$. By Theorem \ref{main2} we deduce the existence of
\begin{itemize}
\item two distinct families of solutions, for $\lambda_1,\la_2$ in a small right neighborhood of $8\pi$, concentrating positively at $\xi_1$ and negatively at $\xi_2$ with either $\xi_1-\xi_2=q_1$ or $\xi_1-\xi_2=q_2$ as $\lambda_1\to 8\pi$ and $\la_2 \to 8\pi $;
\item one family of solutions, for $\lambda_1$, $\la_2$ in a small left neighborhood of $8\pi$, concentrating positively at $\xi_1$ and negatively at $\xi_2$ with $\xi_1-\xi_2=q_3$ as $\lambda_1\to 8\pi$ and $\la_2 \to 8\pi $.
\end{itemize}
\medskip

The case $m_2=0$, namely, as $\la_2\tau^2\to 0^+$, can be also addressed by this approach. Thus, we have that \eqref{mfewt} can be seen as a perturbation of \eqref{mfe}. In this case the nonlinearity $e^{-\tau  u}$ 
is treated as a lower-order term with respect to the main term $e^u$. For simplicity we denote $A(\xi)$ and $B(\xi)$ instead $A_1^*(\xi)$ and $B_1^*(\xi)$ with $m_1=m$ and $\ml J_2=\varnothing$, so that we have the following result.

\begin{theo} \label{main3}
Let $\mathcal{D} \subset \subset \tilde S^m \setminus \Delta$ be a stable critical set of $\varphi_m^*$. Assume that
\begin{equation} \label{condm20}
\hbox{either }A(\xi) >0 \:(<0 \hbox{ resp.)} \qquad \hbox{or} \qquad A(\xi)=0, \:B(\xi)>0 \: (<0 \hbox{ resp.)}
\end{equation}
do hold in a closed neighborhood $U$ of $\mathcal{D}$ in $\tilde S^m \setminus \Delta$. Then, for all $\lambda_1$ in a small right (left resp.) neighborhood of $8 \pi m_1$ and $\lambda_2\tau^2$ in a small right neighborhood of $0$ there is a solution $u_{\lambda_1,\la_2}$ of \eqref{mfewt} which concentrate positively (along sub-sequences) at $m$ points $q_1,\dots, q_m$
$$\frac{\la_1V_1e^{u_{\la_1,\la_2}}}{\int_S V_1e^{u_{\la_1,\la_2}} dv_g}\rightharpoonup 8\pi \sum_{j=1}^{m}\de_{q_j}\quad\text{in measure sense for some $q \in \mathcal{D}$}$$
$$\text{and}\quad \frac{\la_2\tau^{2}V_2e^{ -\tau u_{\la_1,\la_2}}}{\int_S V_2e^{-\tau u_{\la_1,\la_2}} dv_g}\to 0\quad\text{uniformly in $S$. }$$
 
\end{theo}

\noindent Notice that a similar result can be obtained in case $m_1=0$ and $m_2=m$, namely, as $\la_1\to 0^+$ and $\la_2\tau^2\to 8\pi m$, and $u_{\la_1,\la_2}$ concentrates negatively at $m$ different points of $S$. The same conclusion of Theorem \ref{main3} follows: on one hand, under the validity of condition \eqref{condm20} just on $\DD=\{\xi_0\}$, where $\xi_0$ is a non-degenerate local minimum/maximum point of $\varphi_m^*$; and on the other hand, in the special case $|A(\xi)|=O(|\nabla \varphi_m^*(\xi)|_g)$ in a neighborhood of $\DD$ and $B(\xi)>0$ in $\DD$. See proof of Theorem 3.2 and Remark 4.5 in \cite{EF} for more details. Several examples for Theorem \ref{main3} can be derived from each example provided in \cite{EF} for the case $\la_2=0$.

\medskip\noindent
The paper is organized as follows: Some consequences and examples are presented in section \ref{secex}.  In Section \ref{approx}, we construct a first approximation to a solution to \eqref{mfewt} with the required properties and we estimate the size of the error of approximation with appropriate norms. In Section \ref{variat} we describe the scheme of our proofs, by stating the principal results we need, and we give the proof of our Theorem \ref{main2}. Section \ref{sec4} is devoted to the computation of the expansion of the energy functional on the first approximation we constructed in Section \ref{approx}. The proof of Theorem \ref{main3} is done in Section \ref{pthm3}. Sections \ref{appeA} and \ref{appeB} are devoted to prove the intermediate results we state in Section \ref{variat}. 


\section{Consequences and examples}\label{secex}

\noindent In this section we present several consequences of Theorem \ref{main2} and some examples that illustrate our results in the sphere $\mathbb S^2$ and flat two-torus $\mathbb T$. A special case of Theorem \ref{main2} is the following:
\begin{theo} \label{main1}
Let $\mathcal{D} \subset \subset \tilde S^m \setminus \Delta$ be a stable critical set of $\varphi_m^*$. Assume that $A_1^*(\xi) >0$ ($<0$ resp.) and $A_2^*(\xi) >0$ ($<0$ resp.) for all $\xi\in\ml{D}$. Then, for all $\lambda_1$ in a small right (left resp.) neighborhood of $8 \pi m_1$ and $\lambda_2$ in a small right (left resp.) neighborhood of $\dfrac{8 \pi m_2}{\tau^2}$ there is a solution $u_{\lambda_1,\la_2}$ of \eqref{mfewt} which concentrate (along sub-sequences) at $m$ points $q_1,\dots, q_m$ in the sense \eqref{conc} for some $q \in \mathcal{D}$.
\end{theo}

\noindent  The notion of stability we are using here is the following:
\begin{dfn}
\label{stable} A critical set $\mathcal{D} \subset \subset \tilde S^m \setminus \Delta$ of $\varphi_m$ is stable if for any closed neighborhood $U$ of $\mathcal{D}$ in  $\tilde S^m \setminus \Delta$ there exists $\delta>0$ such that, if
$\|G-\varphi_m\|_{C^1(U)}\leq \delta$, then $G$ has at least one critical point in $U$. In particular, the minimal/maximal set of $\varphi_m$ is stable (if $\varphi_m$ is not constant) as well as any isolated c.p. of $\varphi_m$ with non-trivial local degree.
\end{dfn}

\medskip \noindent  Notice that from the definition of $\rho_j$ in \eqref{ro1}-\eqref{ro2} and $A_k^*(\xi)$ in \eqref{vk}, it is readily checked that
$$A_k^*(\xi)= 4\pi \sum_{j\in \ml J_k} \rho_j(\xi_j) [\Delta_g \log V_1(\xi_j)+(-\tau)^{k-1}\frac{8\pi }{|S|} \lf(m_1-{m_2\over \tau}\rg)-2K(\xi_j)],\quad k=1,2$$
 for $\xi$ a c.p. of $\varphi_m^*$, in view of $\nabla \rho_j(\xi_j)=0$ for all $j=1,\dots,m$. If $V_1\ge0$ and $V_2\ge0$ in $S$, then the function $\varphi_2^*$ with $m_1=m_2=1$ always attains its maximum value in $\tilde S^2\sm\Delta$ and the maximal set is clearly stable. Let us stress that $V_1$ and $V_2$ can vanish at some points of $S$. Thus, we have deduced the following fact.
\begin{cor}\label{cor1}
Assume that $V_i\ge 0$ in $S$ for $i=1,2$. If either $\sup_S[2K-\lab\log V_1]<\frac{8\pi}{|S|}\big(1-{1\over \tau}\big)$ or $\inf_S[2K-\lab\log V_1]>\frac{8\pi}{|S|}\big(1-{1\over \tau}\big)$ and either $\sup_S[2K-\lab\log V_2]<\frac{8\pi}{|S|}\big(1- \tau \big)$ or $\inf_S[2K-\lab\log V_2]>\frac{8\pi}{|S|}\big(1-\tau \big)$ then there exist solutions $u_{\la_1,\la_2}$ to \eqref{mfewt} which concentrate at two points, positively at $q_1$ and negatively at $q_2$, 
in the sense \eqref{conc} as $\la_1\to 8\pi$ and $\la_2\tau^2\to 8\pi$, where $(q_1,q_2)$ is a maximum of $\varphi_2^*$ in $\ti S^2\sm\Delta$.
\end{cor}

When $S=\mathbb S^2$ we have that $K=\dfrac{4\pi}{|\mathbb S^2|}$, so that, for $V_1=V_2\equiv 1$ and any $\tau>0$, Corollary \ref{cor1} then provides the existence of blow-up solutions $u_{\la_1,\la_2}$ concentrating at two points as $\la_1\to 8\pi$ and $\la_2\tau^2\to 8\pi$, where $\la_1$ and $\la_2\tau^2$ belongs to a small \emph{left} neighborhood of $8\pi$. In case of a flat two-torus $S=\mathbb T$, $K=0$, so that for $V_1=V_2\equiv 1$ and any $\tau>0$, $\tau\ne 1$, Corollary \ref{cor1} then provides the existence of blow-up solutions $u_{\la_1,\la_2}$ concentrating at two points as $\la_1\to 8\pi$ and $\la_2\tau^2\to 8\pi$, where $\la_1$ belongs to a small \emph{right} (\emph{left} resp.) neighborhood of $8\pi $ if $\tau>1$ ($<1$ resp.) and $\la_2\tau^2$ belongs to a small \emph{left} (\emph{right} resp.) neighborhood of $8\pi$. However, the case $S=\mathbb T$, $V_1=V_2\equiv 1$, $m_1=m_2=1$ and $\tau=1$ is an example for which $A_1^*$ and $A_2^*$ vanishes in $\mathbb T ^2\sm\Delta$ and in particular at c.p.'s.

\medskip Let us mention some examples where $V_1$ and $V_2$ vanish at some points of $S$. Precisely, assume that 
$$V_1(x)=e^{-4\pi \sum\limits_{i=1}^{l_1}n_{1,i}G(x,p_{1,i})}\qquad\text{and}\qquad V_2(x)=e^{-4\pi \sum\limits_{i=1}^{l_2}n_{2,i}G(x,p_{2,i})},$$
with $n_{1,i},n_{2,i}>0$ and $p_{1,i},p_{2,j}\in S$, $i=1,\dots,l_1$ and $j=1,\dots,l_2$ respectively. The zero sets are $\{p_{1,1},\dots,p_{1,l_1}\}$ for $V_1$ and $\{p_{2,1},\dots,p_{2,l_2}\}$ for $V_2$. So, for $m_1=m_2=1$, $m=2$ we have that
$$\varphi_2^*(\xi)=- \sum\limits_{i=1}^{l_1}n_{1,i}G(\xi_1,p_{1,i}) - {1\over \tau^2} \sum\limits_{j=1}^{l_2}n_{2,j}G(\xi_2,p_{2,j})-{2\over \tau}G(\xi_1,\xi_2),$$
and if $\xi$ is a c.p. of $\varphi_2^*$ then
$$A_k^*(\xi)=4\pi\rho_k(\xi_k)\lf[ -\frac{4\pi}{|S|} \sum\limits_{i=1}^{l_k}n_{k,i}+{8\pi\over |S|} \Big(1- \tau^{2k-3}\Big)-2K(\xi_k)\rg],\quad k=1,2.$$
In particular, if $S=\mathbb S^2$  then Corollary \ref{cor1} provides the existence of blow-up solutions $u_{\la_1,\la_2}$ concentrating at two points as $\la_1\to 8\pi$ and $\la_2\tau^2\to 8\pi$ when $\sum_{i=1}^{l_1}n_{1,i}\ne 1-\frac{2}\tau$ and $\sum_{j=1}^{l_2}n_{2,j}\ne 1-2\tau$. We deduce the same conclusion when $S=\mathbb T$ and  $\sum_{i=1}^{l_1}n_{1,i}\ne 2-\frac{2}\tau$ and $\sum_{j=1}^{l_2}n_{2,j}\ne 2-2\tau$. Let us stress that there is no restriction on $n_{1,i},n_{2,j}$'s if $\tau=1$.

\medskip Now, consider the case $m_1=m\ge 2$ and $m_2=1$, namely, $\la_1$ close to $8\pi m $ and $\la_2\tau^2$ close to $8\pi$. Roughly speaking, if $u_{\la_1,\la_2}$ concentrates negatively at $q$ then
$$\la_2\tau\lf(\frac{V_2e^{ -\tau u_{\la_1,\la_2}}}{\int_S V_2e^{-\tau u_{\la_1,\la_2}} dv_g} - {1\over |S|}\rg)\quad\text{ behaves like }\quad 4\pi\cdot {2\over\tau}\lf(\de_q - {1\over |S|}\rg)\quad\text{ as $\la\tau^2\to 8\pi$}$$
and equation \eqref{mfewt}  resembles the singular mean field equation
\begin{equation*}
-\Delta_g v=\la\lf({h e^{v}\over\int_{S}h e^{v} dv_g}-
\frac{1}{|S|}\rg) - 4\pi \al \lf(\de_q - {1\over |S|}\rg) \qquad\text{in $S$},
\end{equation*}
with $\al=\frac2\tau$. According to a result of D'Aprile and Esposito \cite[Theorem 1.4]{DaE}, it follows that the functional
\begin{equation*}\label{fimm1}
\begin{split}
\varphi_{m+1}^* (\xi) = &\ \frac{1}{4\pi}\sum_{j=1}^{m} \log V_1(\xi_j ) + \frac{1}{4\pi \tau^2} \log V_2(\xi_{m+1}) + \sum_{j=1}^{m_1} H(\xi_j,\xi_j) + {1\over\tau^2} H(\xi_{m+1},\xi_{m+1} )\\
& +\sum_{j=1}^{m}\sum_{i=1\atop i\not= j}^{m} G(\xi_i,\xi_j) - {2\over\tau} \sum_{j=1}^{m} G(\xi_j,\xi_{m+1}) ,
\end{split}
\end{equation*}
has a $C^1$-stable critical value for $\xi_{m+1}\in S$ fixed under the assumptions $S\ne \mathbb S^2,\mathbb{RP}^2$ and $\frac{2}\tau\ne 1,\dots,m-1$. Thus, we deduce the next result.

\begin{cor}\label{cor2}
Assume that $V_i>0$ in $S$ for $i=1,2$, $S\ne \mathbb S^2,\mathbb{RP}^2$ and $\dfrac{2}\tau\ne 1,\dots,m-1$. If either $\sup_S[2K-\lab\log V_1]<\frac{8\pi}{|S|}\big(m-{1\over \tau}\big)$ or $\inf_S[2K-\lab\log V_1]>\frac{8\pi}{|S|}\big(m-{1\over \tau}\big)$ and either $\sup_S[2K-\lab\log V_2]<\frac{8\pi}{|S|}\big(1-m\tau \big)$  or $\inf_S[2K-\lab\log V_2]>\frac{8\pi}{|S|}\big(1-m\tau \big)$ then there exist solutions $u_{\la_1,\la_2}$ to \eqref{mfewt} which concentrate at $m+1$ points, positively at $q_1,\dots, q_{m}$ and negatively at $q_{m+1}$, in the sense \eqref{conc} as $\la_1\to 8\pi m$ and $\la_2\tau^2\to 8\pi$, where $(q_1,\dots,q_{m+1})$ is a max-min critical point of $\varphi_{m+1}^*$ in $S^{m+1}\sm\Delta$.
\end{cor}

When $S=\mathbb T$ and $V_1=V_2\equiv 1$, for any $\tau>0$, $m\tau\ne 1$ and $\tau\notin\{2,1,{2\over 3},\dots,{2\over m-1}\}$, Corollary \ref{cor1} then provides the existence of blow-up solutions $u_{\la_1,\la_2}$ concentrating at $m+1$ points as $\la_1\to 8\pi m$ and $\la_2\tau^2\to 8\pi$, where $\la_1$ belongs to a small right (left resp.) neighborhood of $8\pi m$ if $m\tau>1$ ($<1$ resp.) and $\la_2\tau^2$ belongs to a small left (right resp.) neighborhood of $8\pi$. Notice that a similar result can be obtained in case $m_1=1$ and $m_2=m$, namely, $\la_1$ close to $8\pi$ and $\la_2\tau^2$ close to $8\pi m$.

\medskip\noindent 
Observe that on one hand, we generalize existence results of blowing-up solutions for mean field equations \eqref{mfe} in \cite{EF} to an asymmetric problem \eqref{mfewt}. And on the other hand, we perform, in a compact Riemann surface $S$, a similar construction done for a sinh-Poisson equation in bounded domains with Dirichlet boundary conditions by \cite{BaPi} and extended to an asymmetric case in \cite{pr1}. Both problems in \cite{BaPi,pr1} do not contain any potential $V_k$ and the existence of $C^1$-stable critical points of the corresponding $\varphi_m^*$ implies the existence of blowing-up solutions. However, to prove our results is not enough to assume the existence of $C^1$-stable critical points of $\varphi_m^*$ in \eqref{fim}. Admissibility conditions in terms of quantities either $A_k^*$'s or $B_k^*$'s have to be used, in the same spirit of \cite{EF}. 
After completion of this work, we have learned that in \cite{AhBaFi} the existence of $C^1$-stable critical points of vortex type Hamiltonians, including $\varphi_m^*$ in \eqref{fim}, has been proved for a surface $S$ which is not homeomorphic to the sphere nor the projective plane. 

\medskip
Finally, we point out that the type of arguments used to obtain our results have been also developed in several previous works by various authors. Let us quote a few papers from the vast literature concerning singular perturbation problems with nonlinearities of exponential type \cite{ChI,dmr,EMP,EW,FM}. 


\section{Approximation of the solution}\label{approx}
\noindent The main idea to construct approximating solutions of \eqref{mfewt}, as in \cite{EF}, is to use as ``basic cells'' the functions
\begin{equation*}
u_{\delta,\xi}(x)=u_0 \Big(\frac{|x-\xi|}{\delta}\Big)-2\log
\delta, \qquad \de>0,\: \xi\in\R^2,\end{equation*} where $\ds u_0(r)=\log\frac{8}{(1+r^2)^2}.$ They are all the solutions of
\begin{equation*}
\left\{ \begin{array}{ll}\Delta u+e^{u}=0 &\text{in $\R^2$}\\
\int_{\R^2} e^u <\infty, & \end{array} \right.
\end{equation*}
and do satisfy the following concentration property: $e^{u_{\delta,\xi}}\rightharpoonup 8\pi\delta_\xi$
 in measure sense as $\delta \to 0$. We will use now
isothermal coordinates to pull-back $u_{\delta,\xi}$ in $S$. Let us recall that every Riemann surface $(S,g)$ is locally
conformally flat, and the local coordinates in which $g$ is
conformal to the Euclidean metric are referred to as isothermal
coordinates (see for example the simple existence proof provided
by Chern \cite{Chern}). For every $\xi \in S$ it amounts to find a
local chart $y_\xi$, with $y_\xi(\xi)=0$, from a neighborhood of
$\xi$ onto $B_{2r_0}(0)$ (the choice of $r_0$ is independent of
$\xi$) in which $g=e^{\hat \varphi_\xi(y_\xi(x))}dx$, where $\hat
\varphi_\xi \in C^\infty(B_{2r_0}(0),\mathbb{R})$. In particular,
$\hat \varphi_\xi$ relates with the Gaussian curvature $K$ of
$(S,g)$ through the relation:
\begin{equation} \label{equationvarphi}
\Delta \hat \varphi_\xi(y) =-2K(y_\xi^{-1}(y)) e^{\hat
\varphi_\xi(y)} \qquad \hbox{ for }y \in B_{2r_0}(0).
\end{equation}
We can also assume that $y_\xi$, $\hat \varphi_\xi$ depends
smoothly in $\xi$ and that $\hat \varphi_\xi(0)=0$, $\nabla \hat
\varphi_\xi(0)=0$. We now pull-back $u_{\delta,0}$ in $\xi \in S$, for $\delta>0$, by simply setting $\ds U_{\delta,\xi}(x)=u_{\delta,0}(y_\xi(x))=\log \frac{8\delta^2}{(\delta^2+|y_\xi(x)|^2)^2}$ for $x
\in y_\xi^{-1}(B_{2r_0}(0))$. Letting $\chi\in
C_0^\infty(B_{2r_0}(0))$ be a radial cut-off function so that
$0\le\chi\le 1$, $\chi\equiv 1$ in $B_{r_0}(0)$, we introduce the
function $PU_{\de,\xi}$ as the unique solution of
\begin{equation}\label{ePu}
\left\{ \begin{array}{ll} -\Delta_g PU_{\de,\xi} (x)=\chi_\xi(x)
e^{-\varphi_\xi(x)} e^{U_{\de,\xi}(x)}-\frac{1}{|S|}\int_S
\chi_\xi e^{-\varphi_\xi} e^{U_{\de,\xi}} dv_g &\text{in }S\\
\int_S PU_{\de,\xi} dv_g=0,
\end{array}\right.
\end{equation}
where $\chi_\xi(x)=\chi(|y_\xi(x)|)$ and $\varphi_\xi(x)=\hat
\varphi_\xi(y_\xi(x))$. Notice that the R.H.S. in (\ref{ePu}) has
zero average and smoothly depends in $x$, and then (\ref{ePu}) is
uniquely solvable by a smooth solution $PU_{\de,\xi}$.

\medskip \noindent Let us recall the transformation law for $\Delta_g$ under
conformal changes: if $\tilde g=e^{\varphi} g$, then
\begin{equation} \label{laplacian} \Delta_{\tilde g}=e^{-\varphi} \Delta_g.\end{equation}
Decompose now the Green function $G(x,\xi)$, $\xi \in S$, as $\ds G(x,\xi)=-\frac{1}{2\pi} \chi_\xi(x) \log |y_\xi(x)|+H(x,\xi),$ and by (\ref{green}) then deduce that
\begin{equation*}
\left\{ \begin{array}{ll} -\Delta_g H= - \frac{1}{2\pi} \lab
\chi_\xi  \,\log |y_\xi(x)| -\frac{1}{\pi}\langle \grad
\chi_\xi,\grad \log
|y_\xi(x)| \rangle_g-\frac{1}{|S|} &\text{in $S$}\\
\int_S H(\cdot,\xi)\, dv_g=\frac{1}{2\pi} \int_S \chi_\xi \log
|y_\xi(\cdot)| dv_g.&
\end{array} \right.
\end{equation*}
We have used that $\ds \Delta_g \log |y_\xi(x)|= e^{-\hat \varphi_\xi(y)}
\Delta \log|y| \Big|_{y=y_\xi(x)}=2\pi \delta_\xi$ in view of
(\ref{laplacian}). For $r\leq 2r_0$ define
$B_r(\xi)=y_\xi^{-1}(B_r(0))$, $A_{r}(\xi)=B_{r}(\xi) \sm
B_{r/2}(\xi)$, and set
$$f_\xi= {\lab\chi_\xi \over |y_\xi(x)|^2} +2\Big\langle \grad\chi_\xi,\grad |y_\xi(x)|^{-2} \Big\rangle_g+{2\over |S|}
\int_{\mathbb{R}^2} {\chi'(|y|)\over |y|^3}\, dy.$$
Setting $\Psi_{\de,\xi}(x)= PU_{\de,\xi}(x)-\chi_\xi [U_{\delta,\xi}-\log(8\de^2)]-8\pi H(x,\xi),$ by the definition of $f_\xi$ we then have that $-\Delta_g \Psi_{\de,\xi}=-2\de^2 f_\xi+O(\de^4)$ in $S$  so that
$$\int_S f_\xi dv_g=\frac{1}{2\delta^2} \int_S \Delta_g \Psi_{\delta,\xi} dv_g+O(\delta^2)=O(\delta^2)$$
for all $\delta>0$, and hence $\int_S f_\xi dv_g=0$. Therefore,
$F_\xi$ is well defined as the unique solution of
\begin{equation}\label{d2t} \left\{ \begin{array}{ll}-\lab
F_\xi=f_\xi &\text{in }S\\
\int_S F_\xi dv_g=0.&
\end{array}\right. \end{equation}
We have the following asymptotic expansion of $PU_{\de,\xi}$ as
$\delta \to 0$, as shown in \cite{EF}:
\begin{lem}\label{ewfxi}
The function $PU_{\delta,\xi}$ satisfies
$$PU_{\delta,\xi}=\chi_\xi \lf[U_{\delta,\xi}-\log(8\delta^2)\rg]+
8\pi H(x,\xi)+\alpha_{\delta,\xi}-2\delta^2 F_\xi+O(\delta^4|\log
\delta|)$$
uniformly in $S$, where $F_\xi$ is given in \eqref{d2t} and
$$\alpha_{\delta,\xi}=-{4\pi\over|S|} \delta^2 \log \delta +2{\delta^2\over|S|}\lf(\int_{\mathbb{R}^2}
\chi(|y|) \frac{e^{\hat \varphi_\xi(y)}-1}{|y|^2}dy+ \pi-
\int_{\mathbb{R}^2} {\chi'(|y|) \log |y|\over |y| } dy \rg).$$ In
particular, there holds
$$PU_{\delta,\xi}=8\pi G(x,\xi)-2{\de^2 \chi_\xi \over
|y_\xi(x)|^2}+\alpha_{\delta,\xi}-2 \delta^2 F_\xi+O(\delta^4|\log
\delta|)$$
locally uniformly in $S \sm\{\xi\}$.
\end{lem}

\noindent The ansatz will be constructed as follows. Given
$m\in \mathbb{N}$, let us consider distinct points
$\xi_j\in S$ and $\delta_j>0$, $j=1,\dots,m$. In order to have a good
approximation, we will assume that $\exists\, C_0>1\,:$
\begin{equation}\label{repla0}
\delta_j^2= \begin{cases}
\mu_1^2\delta^2 \rho_j(\xi_j) &\text{ for $j\in\{1,\dots,m_1\}$} \\
\mu_2^2\delta^2 \rho_j(\xi_j) &\text{ for $j\in\{m_1+1,\dots,m\}$} 
\end{cases},\ \ \ \text{with } \ 0< \mu_i\le C_0,\quad i=1,2
\end{equation}
\begin{equation}\label{repla1}
|\lambda_1-8\pi m_1 |\le C_0 \de^2|\log
\de| \quad\text{and}\quad|\lambda_2\tau^2-8\pi m_2 |\le C_0 \de^2|\log
\de|,
\end{equation}
where $\de >0$, $m_1\in\{1,\dots,m-1\}$, $m_2=m-m_1$ and $\rho_j$ is given by \eqref{ro1}-\eqref{ro2}. Up to take $r_0$ smaller, we assume
that the points $\xi_j$'s are well separated and $V_1(\xi_j)$, $V_2(\xi_j)$ are uniformly away from zero, namely, we choose
$\xi=(\xi_1,\dots,\xi_{m})\in\Xi$,
where
\begin{equation*} 
\Xi=\{(\xi_1,\dots,\xi_{m}) \in S^{m} \mid
d_g(\xi_i,\xi_j)\geq
4r_0\: \text{ and }\: V_1(\xi_j),\: V_2(\xi_j)\ge r_0\:\:\forall\:i,j=1,\dots,m,\:i\not=j\}.
\end{equation*}
Denote $U_j:= U_{\delta_j,\xi_j}$ and
$W_j=PU_j$, $j=1,\dots,m$, where $P$ is the projection
operator defined by \eqref{ePu}. Thus, our approximating solution
is $\ds W(x)=\sum_{j=1}^{m_1} W_j(x)-{1\over \tau}\sum_{j=m_1+1}^{m} W_j(x)$, parametrized by $(\mu,\xi) \in \ml{M} \times \Xi$, with $\mu=(\mu_1,\mu_2)$ and $\ml{M}=(0,C_0]\times (0,C_0] $. Notice that for $r_0$ small enough we have
that $\ml{D}\subset\Xi\subset \tilde S^{m}\sm\Delta$. We
will look for a solution $u$ of \eqref{mfewt} in the form
$u=W+\phi$, for some small remainder term $\phi$. In terms of
$\phi$, the problem \eqref{mfewt} is equivalent to find $\phi\in
\bar H$ so that
\begin{equation}\label{ephi}
L(\phi)=-[R+N(\phi)] \qquad\text{ in $S$},
\end{equation}
where the linear operator $L$ is defined as
\begin{equation}\label{ol}
L(\phi) = \Delta_g \phi + \sum_{i=1}^2 \lambda_i\tau^{2(i-1)} {V_i(x)e^{(-\tau)^{i-1} W}\over\int_S V_ie^{(-\tau)^{i-1} W}dv_g}\lf(\phi
- {\int_{S} V_ie^{(-\tau)^{i-1} W}\phi dv_g \over\int_S V_ie^{(-\tau)^{i-1} W}dv_g} \rg),
\end{equation}
the nonlinear part $N$ is given by
\begin{equation}\label{nlt}
N(\phi)=N_1(\phi)-N_2(\phi)
\end{equation}
with
\begin{equation}\label{ni}
\begin{split}
N_i(\phi)=&\,\lambda_i \tau^{i-1} \bigg({V_ie^{(-\tau)^{i-1} (W+\phi)}\over\int_S
V_ie^{(-\tau)^{i-1} (W+\phi) }dv_g}-{(-\tau)^{i-1} V_ie^{(-\tau)^{i-1} W}\over\int_S
V_i e^{(-\tau)^{i-1} W}dv_g}\lf[\phi-\frac{\int_{S} V_ie^{(-\tau)^{i-1} W}\phi dv_g}{\int_S
V_i e^{(-\tau)^{i-1} W} dv_g}\rg]\\
&\,-{V_ie^{(-\tau)^{i-1} W}\over\int_S V_i e^{(-\tau)^{i-1} W}dv_g }\bigg)
\end{split}
\end{equation}
for $i=1,2$ and the approximation rate of $W$ is encoded in
\begin{equation}\label{R} R=\Delta_g W+\lambda_1
\lf({V_1(x)e^{W}\over\int_S V_1e^{W}dv_g} - {1\over |S|}\rg)-\lambda_2\tau
\lf({V_2(x)e^{-\tau W}\over\int_S V_2e^{-\tau W}dv_g} - {1\over |S|}\rg).
\end{equation}
Notice that for all $\phi \in \bar H$
$$\int_S L(\phi) dv_g=\int_S N(\phi)dv_g=\int_S R dv_g=0.$$

\noindent In order to get the invertibility of $L$, let us
introduce the weighted norm for any $h\in L^\infty(S)$
\begin{equation*}
\| h \|_*=\sup_{x\in S}
\lf[\sum_{j=1}^{m} \frac{\de_j^\sigma}{(\de_j^2 + \chi_{B_{r_0}(\xi_j)}(x)
|y_{\xi_j}(x)|^2+r_0^2 \chi_{S\setminus
B_{r_0}(\xi_j)}(x))^{1+\sigma/2}}\rg]^{-1} |h(x)|,
\end{equation*}
where $0<\sigma<1$ is a small fixed
constant and $\chi_A$ denotes the characteristic function of the
set $A$. Let us evaluate the approximation rate of $W$ in
$\|\cdot\|_*$ and recall that $m=m_1+ m_2 $:

\begin{lem}\label{estrr0}
Assume \eqref{repla0}-\equ{repla1}. There exists a constant $C>0$,
independent of $\de>0$ small, \st
\begin{equation}\label{re}
\|R\|_*\le  C\left(\delta \,|\nabla\varphi_m^*(\xi)|_g + \de^{2-\sigma}|\log \de|  \right)
\end{equation}
for all $\xi \in \Xi$, where $|\nabla \varphi_m^*(\xi)|_g^2$ stands
for $\displaystyle \sum_{j=1}^m |\nabla_{\xi_j}
\varphi_m^*(\xi)|_g^2$.
\end{lem}
\begin{proof}[\dem] We shall argue in the same way as in \cite[Lemma 2.1]{EF}. First, from  Lemma \ref{ewfxi} we note that for any $j\in\{1,\dots,m\}$, $W_j(x)=U_j(x) - \log (8\delta_j^2) + 8\pi H(x,\xi_j)+O(\de^2 |\log \de|)$ uniformly for $x\in B_{r_0}(\xi_j)$ and $W_j(x)=8\pi G(x,\xi_j) +O(\de^2 |\log \de|)$ uniformly for $x$ on compact subsets of $S \sm\{\xi_j\}$. Since by symmetry and $\hat \varphi_{\xi_j}(0)=0$ we have that 
$$\int_{B_{r_0}(\xi_j)} \rho_j(x) e^{U_j} dv_g=8 \pi \rho_j(\xi_j) + O(\de^2|\log\de|),$$
we then get that for $j\in\{1,\dots, m_1\}$
\begin{eqnarray}\label{iV1eW1}
\int_{B_{r_0}(\xi_j)} V_1e^W dv_g &=& \frac{1}{8\delta_j^2}\int_{B_{r_0}(\xi_j)} V_1 e^{U_j + 8\pi H(x,\xi_j)+8\pi \sum\limits_{l=1,l\ne j}^{m_1} G(x,\xi_l)-{8\pi \over\tau} \sum\limits_{l=m_1+1}^m G(x,\xi_l)+O(\de^2|\log \de |)}dv_g\nonumber\\
&=& {1\over \delta_j^2}[\pi \rho_j(\xi_j) + O(\de^2|\log\de|)] = {\pi \over \mu_1^2\de^2} + O(|\log\de|)
\end{eqnarray}
and for $j\in\{m_1+1,\dots,m\}$
\begin{eqnarray}\label{iV1eW2}
\int_{B_{r_0}(\xi_j)}\hspace{-0.1cm} V_1e^W dv_g&\hspace{-0.3cm}=& \hspace{-0.3cm} \int_{B_{r_0}(\xi_j)} V_1 e^{-{1\over\tau} [U_j -\log(8\de_j^2)+ 8\pi H(x,\xi_j)]+8\pi \sum\limits_{l=1}^{m_1} G(x,\xi_l)-{8\pi \over\tau} \sum\limits_{l=m_1+1,l\ne j}^m G(x,\xi_l)+O(\de^2|\log \de |)}dv_g \nonumber \\
&\hspace{-0.3cm}=&\hspace{-0.2cm}  \int_{B_{r_0}(\xi_j)} V_1(x)\Big[{\rho_j(x)\over V_2(x)} \Big]^{-1/\tau} (\de_j^2+|y_{\xi_j}(x)|^2)^{2/\tau} (1 + O(\de^2|\log \de |))dv_g  \nonumber \\
&\hspace{-0.3cm}=&\hspace{-0.2cm}O(1).
\end{eqnarray}
So, by using \eqref{iV1eW1}-\eqref{iV1eW2} we have that
\begin{eqnarray}\label{iV1eW}
\int_S V_1e^W dv_g =\sum_{j=1}^{m_1} \int_{B_{r_0}(\xi_j)} V_1 e^{W}dv_g + O(1) ={\pi m_1\over \mu_1^2 \de^2} + O(|\log\de|).
\end{eqnarray}
Similarly, for $j\in\{1,\dots, m_1\}$ we get that
\begin{eqnarray}\label{iV2etW2}
\hspace{-0.2cm}\int_{B_{r_0}(\xi_j)} V_2e^{-\tau W} dv_g 
&=& O(1)
\end{eqnarray}
and for $j\in\{m_1+1,\dots,m\}$
\begin{eqnarray}\label{iV2etW1}
\int_{B_{r_0}(\xi_j)} V_2e^{-\tau W} dv_g 
&=& {1\over \delta_j^2}[\pi \rho_j(\xi_j) + O(\de^2|\log\de|)] = {\pi \over \mu_2^2 \de^2} + O(|\log\de|).
\end{eqnarray}
So, by using \eqref{iV2etW2}-\eqref{iV2etW1} we have that
\begin{eqnarray}\label{iV2etW}
\int_S V_2e^{-\tau W} dv_g &=&\sum_{j=m_1+1}^m \int_{B_{r_0}(\xi_j)} V_2 e^{W}dv_g + O(1)=  {\pi m_2\over \mu_2^2 \de^2} + O(|\log\de|).
\end{eqnarray}
By Lemma \ref{ewfxi} and \equ{repla0}, \eqref{iV1eW}, \eqref{iV2etW} we have that
\begin{itemize}
\item in $S \setminus \cup_{j=1}^m B_{r_0}(\xi_j)$ there holds $\la_1 \frac{ V_1 e^W}{\int_S V_1e^W dv_g}=O(\de^2)$  in view of $W(x)=O(1)$;
\item in $B_{r_0}(\xi_j)$, $j\in\{1,\dots,m_1\}$, there holds
\begin{eqnarray*}
\frac{ V_1 e^W}{\int_S V_1e^W dv_g}&=&  \frac{V_1
e^{-\log(8\delta_j^2)+8\pi H(x,\xi_j) + 8\pi\sum\limits_{l=1,l\ne
j}^{m_1}G(x,\xi_l)-{8\pi\over\tau}\sum\limits_{l=m_1+1}^m G(x,\xi_l)+O(\de^2|\log \de|)}}
{\pi m_1 \mu_1^{-2}\de^{-2} + O(|\log\de|)} e^{U_j}\\
&=& \frac{1}{8\pi m_1}\bigg[1+\Big\langle\frac{\nabla (\rho_j \circ
y_{\xi_j}^{-1})(0)}{ \rho_j (\xi_j)},y_{\xi_j}(x)\Big\rangle+O(|y_{\xi_j}(x)|^2+\de^2
|\log \de|)\bigg]  e^{U_j};
\end{eqnarray*}
\item in $B_{r_0}(\xi_j)$, $j\in\{m_1+1,\dots,m\}$, there holds
$$
\frac{ V_1 e^W}{\int_S V_1e^W dv_g} 
= \frac{V_1(x) [\rho_j(x)/V_2(x)]^{-1/\tau}+O(\de^2|\log \de|)} { \pi m_1 \mu_1^{-2} \de^{-2} + O( |\log\de|)} (\de_j^2+|y_{\xi_j}(x)|^2)^{2/\tau}= O(\de^2).$$
\end{itemize}
Similarly as above, we have that
\begin{itemize}
\item in $S \setminus \cup_{j=1}^m B_{r_0}(\xi_j)$ there holds $\la_2\tau \frac{ V_2 e^{-\tau W} }{\int_S V_2 e^{-\tau W} dv_g}=O(\de^2)$  in view of $W(x)=O(1)$;
\item in $B_{r_0}(\xi_j)$, $j\in\{1,\dots,m_1\}$, there holds
$$\frac{ V_2 e^{-\tau W} }{\int_S V_2e^{-\tau W} dv_g}
=\frac{V_2(x) [\rho_j(x) / V_1(x) ]^{-\tau}+O(\de^2|\log \de|)}{ \pi m_2 \mu_2^{-2} \de^{-2} + O( |\log\de|)} 
(\de_j^2+|y_{\xi_j}(x)|^2)^{2\tau} = O(\de^2),$$
\item in $B_{r_0}(\xi_j)$, $j\in\{m_1+1,\dots,m\}$, there holds
$$\frac{ V_2 e^{-\tau W} }{\int_S V_2e^{-\tau W} dv_g}=\frac{1}{8\pi m_2 }\bigg[1+\Big\langle\frac{\nabla ( \rho_j \circ
y_{\xi_j}^{-1})(0)}{ \rho_j(\xi_j)},y_{\xi_j}(x)\Big\rangle+O(|y_{\xi_j}(x)|^2+\de^2
|\log \de|)\bigg]  e^{U_j}.$$
\end{itemize}
Since as before
$$\int_S \chi_j e^{-\varphi_j} e^{U_j} dv_g=\int_{B_{r_0}(0)} {8 \delta_j^2\over (\delta_j^2 +
|y|^2 )^2}  dy+O(\delta^2)=8\pi + O(\de^2)$$
with $\varphi_j=\varphi_{\xi_j}$, for $R$ given by \eqref{R} we then have that
\begin{eqnarray*}
R&=&-\sum_{j=1}^{m_1} \chi_j e^{-\varphi_j}e^{U_j} +  {\la_1V_1e^{W}\over\int_S V_1e^{W}dv_g}+{8\pi m_1-\la_1\over |S|} + O(\de^2)\\
&&+{1\over\tau} \sum_{j=m_1+1}^{m} \chi_j e^{-\varphi_j}e^{U_j} -  {\la_2\tau V_2e^{-\tau W}\over\int_S V_2e^{-\tau W}dv_g}+{\la_2\tau^2- 8\pi m_2\over |S|\tau} + O(\de^2),
\end{eqnarray*}
where $\chi_j=\chi_{\xi_j}$. By previous computations we now deduce that $R(x)=O(\de^2)$ in $S \setminus \cup_{j=1}^m B_{r_0}  (\xi_j)$,
\begin{eqnarray*}
R&=&  \lf[-e^{-\varphi_j}+{\la_1\over 8\pi m_1}+O\big( |\nabla\log ( \rho_j \circ y_{\xi_j}^{-1})(0)||y_{\xi_j}(x)|+|y_{\xi_j}(x)|^2+\de^2 |\log \de|\big)\rg] e^{U_j}\\
&&+O(|\la_1-8\pi m_1|+|\la_2\tau^2- 8\pi m_2 |+\de^2)\\
&=& e^{U_j}O\lf(|\nabla \log( \rho_j \circ y_{\xi_j}^{-1})(0)|
|y_{\xi_j}(x)|+|y_{\xi_j}(x)|^2+|\la_1-8\pi m_1|+\de^2|\log \de|\rg)\\
&& +O(|\la_1-8\pi m_1| +|\la_2\tau^2- 8\pi m_2|+ \de^2)
\end{eqnarray*}
in $B_{r_0}  (\xi_j)$, $j\in\{1,\dots,m_1\}$ and similarly,
\begin{eqnarray*}
R&=& e^{U_j}O\lf(|\nabla \log( \rho_j \circ y_{\xi_j}^{-1})(0)|
|y_{\xi_j}(x)|+|y_{\xi_j}(x)|^2+|\la_2\tau^2- 8\pi m_2|+\de^2|\log \de|\rg) \\
&&+O(|\la_1-8\pi m_1| +|\la_2\tau^2- 8\pi m_2| + \de^2)
\end{eqnarray*}
in $B_{r_0}  (\xi_j)$, $j\in\{m_1+1,\dots,m\}$, in view of $\varphi_j(\xi_j)=0$ and $\nabla \varphi_j(\xi_j)=0$. From the definition of $\|\cdot \|_*$ and \eqref{repla1} we deduce the validity of \eqref{re}. This finishes the proof.
\end{proof}


\section{Variational reduction and proof of main results}\label{variat}
\noindent The solvability theory for the linear operator $L$
given in (\ref{ol}), obtained as the linearization of
\eqref{mfewt} at the approximating solution $W$, is a key step in the so-called nonlinear Lyapunov-Schimdt reduction. Notice that formally the operator $L$ approaches $\hat L$ defined in $\mathbb{R}^2$ as
$$\hat L(\phi) = \Delta\phi+{8\over (1+|y|^2)^2}\lf(\phi -
{1\over\pi}\int_{\R^2}{\phi(z)\over (1+|z|^2)^2}\,dz\rg),$$
by setting $y =y_{\xi_j}(x)/\delta_j $ as $\de \to 0$. Due to the intrinsic invariances, the kernel of $\hat L$ in $L^\infty(\R^2)$  is non-empty and is spanned by $1$ and $Y_j$, $j=0,1,2$, where $\ds Y_{i}(y) = { 4 y_i \over 1+|y|^2}$, $i=1,2$, and $\ds Y_{0}(y) = 2\,{1-|y|^2\over 1+|y|^2}$. 
Since \cite{DeKM,EF,EGP} it is by now rather standard to show the invertibility of $L$ in a suitable ``orthogonal" space, and a sketched proof of it will be given in Appendix A. However, as it was observed in \cite{EF}, for Dirichlet Liouville-type equations on bounded domains as in \cite{DeKM,EGP}, the corresponding limiting operator $\tilde L$ takes the form $\tilde L(\phi)=\Delta\phi+{8\over (1+|y|^2)^2}\phi$ and the function $1$ does not belong to its kernel, making possible to disregard the ``dilation parameters" $\de_i$ in the reduction. As we will see, two additional parameters $\mu_1$ and $\mu_2$ are needed in the reduction (one associated to all ``positive bubbles'' and the other one to all ``negative bubbles'') and in this respect our problem displays a new feature w.r.t. Dirichlet Liouville-type equations, making our situation very similar to the one arising in the study of critical problems in higher dimension. Roughly speaking, $L$ resemble a ``direct sum'' of linear operators for mean field type equations.

\medskip \noindent To be more precise, for $i=0,1,2$ and $j=1,\dots,m$ introduce the functions
\begin{equation*}
Z_{ij}(x) = Y_i\lf({y_{\xi_j}(x)\over \delta_j}\rg)=\left\{\begin{array}{ll}
\ds 2  {\de_j^2- |y_{\xi_j}(x)|^2\over \de_j^2+|y_{\xi_j}(x)|^2} &\hbox{for }i=0\\[0.5cm]
\ds {4\de_j(y_{\xi_j}(x))_i \over \de_j^2+|y_{\xi_j}(x)|^2}&\hbox{for }i=1,2, \end{array} \right.
\end{equation*}
and set $Z_1=\displaystyle \sum_{l=1}^{m_1} Z_{0l}$ and $Z_2=\displaystyle \sum_{l=m_1+1}^m Z_{0l}$. For $i=1,2$ and $j=1,\dots,m$, let $PZ_i$, $PZ_{ij}$ be the projections of
$Z_i$, $Z_{ij}$ as the solutions in $\bar H$ of
\begin{equation} \label{ePZ}
\begin{array}{rl}
 \Delta_g PZ_i&\ds=\chi_j \Delta_g Z_i-\frac{1}{|S|}\int_S
\chi_j \Delta_g Z_i dv_g\\[0.5cm]
\Delta_g PZ_{ij} &\ds=\chi_j \Delta_g Z_{ij}-\frac{1}{|S|}\int_S
\chi_j  \Delta_g Z_{ij} dv_g.
\end{array}
\end{equation}
In Appendix A we prove the following result:
\begin{prop} \label{p2}
There exists $\delta_0>0$ so that for all $0<\delta\leq \delta_0$,
$h\in C(S)$ with $\int_Sh\,dv_g=0$, $\mu\in\ml{M}$, $\xi \in \Xi$ there is a
unique solution $\phi \in \bar H \cap W^{2,2}(S)$ and
$c_{0i},c_{ij} \in \mathbb{R}$ of
\begin{equation}\label{plco}
\left\{ \begin{array}{ll}
L(\phi) = h + \displaystyle \sum_{i=1}^{2}\Big[ c_{0i} \Delta_gPZ_i + \sum_{j=1}^m c_{ij} \Delta_g PZ_{ij}\Big]&\text{in }S\\
\ds\int_S \phi \Delta_g PZ_i dv_g=\int_S \phi \Delta_g PZ_{ij} dv_g=0 &\forall\: i=1,2,\, j=1,\dots,m.
\end{array} \right.
\end{equation}
Moreover, the map $(\mu ,\xi) \mapsto (\phi,c_{0i},c_{ij})$ is
twice-differentiable in $\mu$ and one-differentiable in $\xi$ with
\begin{eqnarray}
&&\|\phi \|_\infty \le C |\log \de| \|h\|_*\:,\qquad \displaystyle \sum_{i=1}^{2}\Big[|c_{0i}|+ \sum_{j=1}^m  |c_{ij}| \Big] \le C\|h\|_* \label{estmfe1} \\
&& \sum_{i=1}^2 \bigg[\|\fr_{\mu_i} \phi\|_\infty + \sum_{k=1}^2{1\over |\log \de|} \|\fr_{\mu_i\mu_k} \phi \|_\infty +  \sum_{j=1}^m \de \|\fr_{(\xi_j)_i} \phi\|_\infty  \bigg] \le C |\log \de|^2 \|h\|_*\label{estd}
\end{eqnarray}
for some $C>0$.
\end{prop}

\medskip \noindent Let us recall that  $u=W+\phi$  solves \eqref{mfewt} if $\phi\in \bar H$ does satisfy (\ref{ephi}). Since the operator $L$ is not fully invertible, in view of Proposition \ref{p2} one can solve the nonlinear problem (\ref{ephi}) just up to a linear combination of $\Delta_g PZ_1$, $\Delta_g PZ_2$ and $\Delta_g PZ_{ij}$, as explained in the following (see Appendix B the proof):
\begin{prop}\label{lpnlabis}
There exists $\delta_0>0$ so that for
all $0<\delta\leq \delta_0$, $\mu\in\ml{M}$, $\xi \in \Xi$ problem
\begin{equation}\label{pnlabis}
\left\{ \begin{array}{ll}
L(\phi)= -[R+N(\phi)] +  \displaystyle \sum_{i=1}^{2}\Big[ c_{0i} \Delta_gPZ_i + \sum_{j=1}^m c_{ij} \Delta_g PZ_{ij}\Big] & \text{in } S\\
\ds \int_S \phi \Delta_g PZ_i dv_g=\int_S \phi \Delta_g PZ_{ij} dv_g= 0 &\forall\:
i=1,2,\, j=1,\dots,m
\end{array} \right.
\end{equation}
admits a unique solution $\phi(\mu ,\xi) \in \bar H \cap
W^{2,2}(S)$ and $c_{0i} (\mu,\xi),\,c_{ij}(\mu ,\xi) \in \R$, $i=1,2$
and $j=1,\dots,m$, where $\de_j>0$ are as in \eqref{repla0} and
$N$, $R$ are given by \eqref{nlt}, \eqref{R}, respectively. Moreover, the map
$(\mu,\xi)\mapsto
(\phi(\mu,\xi),c_{0i}(\mu,\xi),c_{ij}(\mu,\xi))$ is
twice-differentiable in $\mu$ and one-differentiable in $\xi$ with
\begin{eqnarray} \label{cotaphi1bis}
 \|\phi\|_\infty &\hspace{-0.2cm} \le & \hspace{-0.2cm} C\left( \delta |\log \delta |\,  |\nabla \varphi_m(\xi)|_g+ \delta^{2-\sigma}|\log\delta|^2\right)\\
\label{cotadphi1bis}
\hspace{-0.6cm}\sum_{i=1}^2\Big[ \|\fr_{\mu_i} \phi \|_\infty +  \sum_{j=1}^m \de \|\fr_{(\xi_j)_i} \phi \|_\infty +\sum_{k=1}^2{\|\fr_{\mu_i\mu_k} \phi \|_\infty\over |\log \de|} \Big] &\hspace{-0.2cm} \le &\hspace{-0.2cm}C \left(\de |\log \delta |^2 |\nabla \varphi_m(\xi)|_g+\delta^{2-\sigma}|\log\delta|^3\right) 
\end{eqnarray}
\end{prop}
\noindent The function $[W+\phi](\mu,\xi)$ will be a true solution of \eqref{ephi} if $\mu\in\ml M$ and $\xi\in \Xi$ are such that
$c_{0i}(\mu,\xi)=c_{ij}(\mu, \xi)=0$ for all $i=1,2,$ and $j=1,\dots,m$. This problem is equivalent to finding critical
points of the reduced energy $E_{\lambda_1,\la_2}(\mu, \xi)= J_{\lambda_1,\la_2}\big([W+\phi](\mu,\xi)\big)$, where $J_{\lambda_1,\la_2}$ is given by \eqref{energy}, as stated in (we omit its proof):
\begin{lem}\label{cpfc0bis}
There exists $\delta_0$ such that, if $(\mu,\xi)\in
\ml{M}\times \Xi$ is a critical point of $E_{\lambda_1,\la_2}$ for $0< \de\le \de_0$, then
$u=W(\mu,\xi)+\phi(\mu ,\xi)$ is a solution to \eqref{mfewt}, where
$\de_i$ are given by \eqref{repla0}.
\end{lem}
\noindent Once equation \eqref{mfewt} has been reduced to the search of c.p.'s for $E_{\lambda_1,\la_2} $, it becomes crucial to show that the main asymptotic term of $E_{\lambda_1,\la_2} $ is given by $J_{\lambda_1,\la_2} (W)$, for which an expansion has been given in Theorem \ref{expansionenergy}. More precisely, by estimates in Appendix B we have that
\begin{theo} \label{fullexpansionenergy}
Assume \eqref{repla0} - \eqref{repla1}. The following expansion does
hold
\begin{eqnarray} \label{fullJUt}
E_{\lambda_1,\la_2} (\mu,\xi) &=&-8\pi \Big(m_1 + {m_2\over\tau^2}\Big) - \lambda_1 \log (\pi m_1) - \la_2\log(\pi m_2) + 2\big(\lambda_1 -8\pi m_1 \big)\log\de \nonumber \\
&&\,\, + {2\over \tau^2} \big(\lambda_2\tau^2 -8\pi m_2 \big)\log\de -32\pi^2 \varphi_m^*(\xi)+ 2\big(\lambda_1 -8\pi m_1 \big)\log\mu_1 \\
&&\,\, + A_1^*(\xi) \mu_1^2\delta^2\log  \delta + \lf[A_1^*(\xi) \mu_1^2 \log\mu_1- B_1^*(\xi)\mu_1^2\rg] \delta^2 + {1\over\tau^2}\Big\{ 2\big(\la_2\tau^2 -8\pi m_2 \big) \log\mu_2 \nonumber \\
&&\,\, + A_2^*(\xi) \mu_2^2 \de^2\log
 \de +  \lf[A_2^*(\xi) \mu_2^2\log\mu_2 - B_2^*(\xi) \mu_2^2\rg] \de^2\Big\} \nonumber  + o(\de^2) +r_{\lambda_1,\la_2}(\mu,\xi) \nonumber
\end{eqnarray}
in $C^2(\mathbb{R}^2)$ and $C^1(\Xi)$ as $\de\to 0^+$, where
$\varphi_m^*(\xi)$, $A_k^*(\xi)$ and $B_k^*(\xi)$, $k=1,2$  are given by \eqref{fim}, \eqref{vk} and \eqref{Bk}, $k=1,2$, respectively. The term $r_{\lambda_1,\la_2}(\mu,\xi)$ satisfies
\begin{eqnarray} \label{rlambda}
|r_{\lambda_1,\la_2}(\mu,\xi)|&+&\frac{\de |\nabla_\xi r_{\lambda_1,\la_2}(\mu,\xi)|}{|\log \de|} +\frac{|\nabla_\mu r_{\lambda_1,\la_2}(\mu,\xi)|}{|\log \de|} \nonumber \\ 
&+&\frac{|D^2_{\mu} r_{\lambda_1,\la_2}(\mu,\xi)|}{|\log \de|^2}  \leq C\lf( \delta^2 |\log \delta |\, |\nabla \varphi_m^*(\xi)|_g^2 + \de^{3-\sigma}|\log\de|^2\rg)
\end{eqnarray}
for some $C>0$ independent of $(\mu,\xi)\in \ml M\times\Xi$.
\end{theo}
\noindent We are now in position to establish the main result stated in the Introduction. We shall argue similarly to \cite[Theorem 1.5]{EF}.
\begin{proof}[{\bf Proof (of Theorem \ref{main2}):}]  According to Lemma \ref{cpfc0bis}, we just need to find a critical point of $E=E_{\la_1,\la_2}(\mu,\xi)$ with $\mu=(\mu_1,\mu_2)$. Recall that $\tau>0$ is fixed. Assumptions \eqref{cond0} and \eqref{cond1} allows us to choose $\mu_k=\mu_k(\la_k,\xi)$ for $\la_k\tau^{2(k-1)}$ close to $8\pi m_k$, $k=1,2$, respectively. Precisely, fixing $k\in\{1,2\}$ we choose $\la_k\tau^{2(k-1)}-8\pi m_k=\de^2$ ($-\de^2$ resp.) if either $A_k^*(\xi)>0$ ($<0$ resp.) or $A^*(\xi)=0$, $B_k^*(\xi)>0$ ($<0$ resp.) in $U$. Thus, we deduce the expansions
\begin{equation*}
\begin{split}
\frac{ \tau^{2(k-1)} \fr_{\mu_k} E(\mu,\xi)}{\la_k\tau^{2(k-1)}-8\pi m_k} =&\, \frac{2}{\mu_k} + 2 A_k^*(\xi) \mu_k \log\de+A_k^*(\xi)(2\mu_k\log\mu_k+\mu_k)-2B_k^*(\xi)\mu_k\\
&\,+ o(1) + O\lf(|\log\de |^2\,|\grad\varphi_m^*(\xi)|^2_g\rg)
\end{split}
\end{equation*}
and
\begin{equation*}
\begin{split}
\frac{ \tau^{2(k-1)} \fr_{\mu_k\mu_k} E(\mu,\xi)}{\la_k\tau^{2(k-1)}-8\pi m_k} =&\, -\frac{2}{\mu_k^2} + 2 A_k^*(\xi)\log\de +A_k^*(\xi)(2\log\mu_k+3) -2B_k^*(\xi)  \\
&\,+ o(1) + O\lf(|\log \de |^3\,|\grad\varphi_m^*(\xi)|^2_g\rg),
\end{split}
\end{equation*}
as $\de\to 0^+$. Arguing in the same way as in the proof of Theorem 3.2 in \cite{EF}, we conclude existence of a $C^1$ map $\mu_k=\mu_k(\la_k,\xi)$ satisfying $\fr_{\mu_k} E(\mu(\la,\xi),\xi)=0$, with $\la=(\la_1,\la_2)$ and $\mu=(\mu_1,\mu_2)$ for all $\xi\in U$. Now, considering $\tilde E (\xi)=E_{\la}\big(\mu_1(\lambda_1,\xi),\mu_2(\la_2,\xi),\xi\big)$ and again arguing in the same way as in the proof of Theorem 3.2 in \cite{EF} it follows that $\tilde E (\xi)=-32\pi^2
\varphi_m^*(\xi)+ O(\de^2|\log\de|)$,
\begin{equation*}
\begin{split}
\nabla_\xi \tilde E (\xi)& = \nabla_\xi E \big(\mu_1 (\lambda_1,\xi),\mu_2(\la_2,\xi),\xi\big)+\grad_\mu E \big( \mu_1(\lambda_1,\xi),\mu_2(\la_2,\xi),\xi\big)\grad_\xi \mu(\lambda,\xi)\\
&=-32\pi^2 \nabla \varphi_m^*(\xi)+ O(\de |\log \de|^{2})
\end{split}
\end{equation*}
uniformly in $\xi \in U$ and there exists a
critical point $\xi_{\lambda_1,\la_2}=\xi_\de \in U$ of $\tilde E (\xi)$, since $\DD$ is a stable critical set of $\varphi_m^*$ (see Definition \ref{stable}).  
Up to take $U$ smaller so that $\grad\varphi_m^*(\xi)\ne 0$ for all $\xi\in U\sm\DD$, it can be deduced that the pair $\big(\mu(\lambda_1,\la_2,\xi_\de),\xi_\de\big)$ is a c.p. of $ E(\mu,\xi)$ and, along a sub-sequence, $\xi_\de \to q \in \DD$ as $\de\to 0$, namely, as $\la_1 \to 8\pi m_1$ and $\la_2\tau^2\to 8\pi m_2$. By construction, the corresponding solution has the required asymptotic properties \eqref{conc}. See proof of Theorem 1.5 in \cite{EF} for more details. This completes the proof.
\end{proof}


\section{The reduced energy}\label{sec4}

\noindent The purpose of this section is to give an asymptotic
expansion of the ``reduced energy" $J_{\la_1,\la_2}(W)$, where
$J_{\la_1,\la_2}$ is the energy functional given by \eqref{energy}. For
technical reasons, we will be concerned with establishing it in a
$C^2$-sense in $\mu$ and just in a $C^1$-sense in $\xi$. To this
aim, the following result will be very useful, see \cite[Lemma 3.1]{EF} for a proof.

\begin{lem}\label{ieuf}
Letting $f\in C^{2,\gamma}(S)$ (possibly depending in $\xi$),
$0<\gamma<1$, denote as $P_2(f)$ the second-order Taylor expansion
of $f(x)$ at $\xi$:
$$P_2 f(x)=f(\xi)+\langle\grad (f \circ y_\xi^{-1}) (0), y_\xi(x)\rangle+{1\over2}\langle
D^2 (f\circ y_\xi^{-1})(0)y_\xi(x), y_\xi(x) \rangle.$$ The
following expansions do hold as $\delta \to 0$:
\begin{eqnarray*}
\int_S \chi_\xi e^{-\varphi_\xi} f(x) e^{U_{\de,\xi}} dv_g&=& 8\pi f(\xi)-2 \de^2 \Delta_g f (\xi) \left[ 2\pi \log \delta+ \int_{\mathbb{R}^2} {\chi'(|y|) \log |y|\over |y| } dy +\pi \right]\\
&&+8\de^2\int_S \chi_\xi e^{-\varphi_\xi} {f(x)-P_2(f)(x)\over
|y_{\xi}(x)|^4}\,dv_g+ 4\de^2 f(\xi) \int_{\mathbb{R}^2}
{\chi'(|y|) \over |y|^3 } dy+o(\delta^2),
\end{eqnarray*}
\begin{eqnarray*}
\int_S \chi_\xi e^{-\varphi_\xi} f(x)
e^{U_{\de,\xi}}\frac{dv_g}{\delta^2+|y_\xi(x)|^2} =
\frac{4\pi}{\delta^2}f(\xi)+\pi \Delta_g f(\xi)+O(\delta^{\gamma})
\end{eqnarray*}
and
\begin{eqnarray*}
\int_S \chi_\xi e^{-\varphi_\xi} f(x) e^{U_{\de,\xi}}\frac{a
\delta^2-|y_\xi(x)|^2}{(\delta^2+|y_\xi(x)|^2)^2} dv_g =\frac{4
\pi}{3 \de^2}(2a-1) f(\xi)+(a-2)\frac{\pi}{3} \Delta_g f(\xi)
+O(\delta^\gamma)
\end{eqnarray*}
for $a \in \mathbb{R}$.
\end{lem}

\medskip \noindent We are now ready to establish the expansion of $J_{\la_1,\la_2}(W)$:
\begin{theo} \label{expansionenergy}
Assume \eqref{repla0}-\eqref{repla1}. The following expansion does
hold
\begin{eqnarray} \label{JUt}
J_{\lambda_1,\la_2} (W) &=&\,\, -8\pi \Big(m_1+{m_2\over\tau^2}\Big) -\lambda_1 \log (\pi
m_1)-\la_2\log(\pi m_2)  + 2\big(\lambda_1 -8\pi m_1 \big)\log\de \nonumber\\
&&\,\, + {2\over \tau^2} \big(\lambda_2\tau^2 -8\pi m_2 \big)\log\de -32\pi^2 \varphi_m^*(\xi)+ 2\big(\lambda_1 -8\pi m_1 \big)\log\mu_1\\
&&\,\, + A_1^*(\xi) \mu_1^2\delta^2\log  \delta + \lf[A_1^*(\xi) \mu_1^2 \log\mu_1- B_1^*(\xi)\mu_1^2\rg] \delta^2+ {1\over\tau^2}\Big\{ 2\big(\la_2\tau^2 -8\pi m_2 \big) \log\mu_2 \nonumber \\
&&\,\, + A_2^*(\xi) \mu_2^2 \de^2\log
 \de +  \lf[A_2^*(\xi) \mu_2^2\log\mu_2 - B_2^*(\xi) \mu_2^2\rg] \de^2\Big\} + o(\de^2)\nonumber
\end{eqnarray}
in $C^2(\mathbb{R}^2)$ and $C^1(\Xi)$ as $\de \to 0^+$, where
$\varphi_m^*(\xi)$, $A_1^*(\xi)$, $A_2^*(\xi)$, $B_{1}^*(\xi)$ and $B_{2}^*(\xi)$ are given by \eqref{fim}, \eqref{vk} and \eqref{Bk}, $k=1,2$, respectively.
\end{theo}

\noindent As in \cite[Theorem 3.2]{EF}, the proof will be divided into several steps.
\begin{proof}[{\bf Proof (of (\ref{JUt}) in $C(\mathbb{R}^2 \times \Xi)$):}]
First, let us consider the term. Integrating by parts we have that
\begin{equation*}
\begin{split}
\int_S |\grad W|_g^2 dv_g 
&= \sum_{j,l=1}^{m_1} \int_S \chi_j e^{-\varphi_j} e^{U_j}W_l dv_g -{1\over \tau} \sum_{j=1}^{m_1}\sum_{l=m_1+1}^m \int_S \chi_j e^{-\varphi_j} e^{U_j}W_l dv_g\\
&\ \ \ -{1\over \tau} \sum_{j=m_1+1}^m\sum_{l=1}^{m_1} \int_S \chi_j e^{-\varphi_j} e^{U_j}W_l dv_g + {1\over \tau^2} \sum_{j,l=m_1+1}^m \int_S \chi_j e^{-\varphi_j} e^{U_j}W_l dv_g
\end{split}
\end{equation*}
in view of $\int_S W dv_g=0$. Since by (\ref{green}) and (\ref{ePu})
\begin{eqnarray} \label{tricky}
\int_S \chi_j e^{-\varphi_j} e^{U_j} G(x,\xi_l) dv_g= \int_S (-\Delta_g PU_j) G(x,\xi_l) dv_g= PU_j(\xi_l)
\end{eqnarray}
for all $j,l=1,\dots,m$, by Lemmata \ref{ewfxi}, \ref{ieuf}, (\ref{tricky}) and computations done in the proof of \cite[Theorem 3.2]{EF}, we have that for $l=j$
\begin{eqnarray*}
\int_S \chi_j e^{-\varphi_j}e^{U_j}W_j dv_g=-16\pi -32 \pi \log \de_j +64\pi^2 H(\xi_j,\xi_j)+16\pi \alpha_{\de_j,\xi_j}-32\pi \delta_j^2 F_{\xi_j}(\xi_j)+O(\de^4 |\log \delta|^2).
\end{eqnarray*}
Similarly, by Lemmata \ref{ewfxi}, \ref{ieuf}  and (\ref{tricky}) we have that for $l \not= j$
\begin{eqnarray*}
\int_S \chi_j e^{-\varphi_j}e^{U_j}W_l dv_g
&=& 64\pi^2 G(\xi_l,\xi_j) +8\pi (\alpha_{\de_j,\xi_j}+\alpha_{\de_l,\xi_l})-16\pi (\delta_j^2 F_{\xi_j}(\xi_l)+ \delta_l^2 F_{\xi_l}(\xi_j) \\&&+O(\de^4 |\log \delta|^2).
\end{eqnarray*}
Setting $\ds\alpha_{1,\de,\xi}=\sum_{j=1}^{m_1} \alpha_{\de_j,\xi_j}$, $\ds \alpha_{2,\de,\xi}=\sum_{j=m_1+1}^{m} \alpha_{\de_j,\xi_j},$ $\ds F_{1,\de,\xi}(x)=\sum_{j=1}^{m_1} \delta_j^2
F_{\xi_j}(x)$ and $\ds F_{2,\de,\xi}(x)=\sum_{j=m_1+1}^{m} \delta_j^2
F_{\xi_j}(x)$, we find that
\begin{eqnarray*}
&&\sum_{j,l=1}^{m_1} \int_S \chi_j e^{-\varphi_j} e^{U_j}W_l dv_g=-16\pi m_1 + \sum_{j=1}^{m_1}\Big[ -32\pi \log(\mu_1 \de) - 16\pi \log V_1(\xi_j) - 64\pi^2 H(\xi_j,\xi_j) \\
&&- 64\pi^2\sum_{i=1\atop i\ne j}^{m_1} G(\xi_j,\xi_i) + {128\pi^2\over \tau} \sum_{i=m_1+1}^m G(\xi_j,\xi_i)\Big]   + 16\pi m_1\al_{1,\de,\xi} - 32\pi \sum_{j=1}^{m_1} F_{1,\de,\xi}(\xi_j) + O(\de^4|\log\de|^2),
\end{eqnarray*}
\begin{eqnarray*}
\sum_{j=1}^{m_1}\sum_{l=m_1+1}^m \int_S \chi_j e^{-\varphi_j} e^{U_j}W_l dv_g & = & 64\pi^2 \sum_{j=1}^{m_1} \sum_{l=m_1+1}^{m} G(\xi_j,\xi_l)  + 8\pi m_2 \al_{1,\de,\xi} + 8\pi m_1 \al_{2,\de,\xi}\\
&& -16\pi \sum_{j=m_1+1}^{m} F_{1,\de,\xi}(\xi_j) -16\pi \sum_{j=1}^{m_1} F_{2,\de,\xi}(\xi_j) + O(\de^4|\log\de|^2),
\end{eqnarray*}
\begin{eqnarray*}
\sum_{j=m_1+1}^{m}\sum_{l=1}^{m_1} \int_S \chi_j e^{-\varphi_j} e^{U_j}W_l dv_g & = & 64\pi^2 \sum_{j=m_1+1}^{m} \sum_{l=1}^{m_1} G(\xi_j,\xi_l)  + 8\pi m_2 \al_{1,\de,\xi} + 8\pi m_1 \al_{2,\de,\xi}\\
& & -16\pi \sum_{j=m_1+1}^{m} F_{1,\de,\xi}(\xi_j) -16\pi \sum_{j=1}^{m_1} F_{2,\de,\xi}(\xi_j) + O(\de^4|\log\de|^2),
\end{eqnarray*}
and
\begin{eqnarray*}
&&\hspace{-0.6cm}\sum_{j,l=m_1+1}^{m} \int_S \chi_j e^{-\varphi_j} e^{U_j}W_l dv_g=-16\pi m_2 + \sum_{j=m_1+1}^{m}\Big[ -32\pi \log(\mu_2 \de) - 16\pi \log V_2(\xi_j) - 64\pi^2 H(\xi_j,\xi_j) \\
&&\hspace{-0.6cm} - 128\pi^2\tau \sum_{i=1}^{m_1} G(\xi_j,\xi_i) - 64\pi^2 \sum_{i=m_1+1\atop i\ne j}^m G(\xi_j,\xi_i)\Big]   + 16\pi m_2\al_{2,\de,\xi} - 32\pi \sum_{j=m_1+1}^{m} F_{2,\de,\xi}(\xi_j) + O(\de^4|\log\de|^2)
\end{eqnarray*}
in view of (\ref{repla0}). 
Now, setting $\ds\alpha_{\de,\xi}=\alpha_{1,\de_j,\xi_j} - {1\over\tau} \al_{2,\de,\xi}$ and $\ds F_{\de,\xi}(x) = F_{1,\de,\xi}(x) - {1\over \tau} F_{2,\de,\xi}(x),$ summing up the four previous expansions, for the gradient term we get that
\begin{eqnarray}\label{gtJW}
{1\over 2}\int_S  |\grad W|_g^2 dv_g & = & -\, 8 \pi \lf(m_1+{m_2\over \tau^2 }\rg) -16 \pi \lf(m_1 \log(\mu_1\de)+{m_2\over \tau^2}\log(\mu_2\de) \rg) - 32 \pi^2 \varphi_m^*(\xi)\\
&&+\, 8 \pi \lf(m_1 - {m_2\over \tau}\rg) \alpha_{\de,\xi} -16 \pi \sum_{j=1}^{m _1}F_{\delta,\xi}(\xi_j) + {16 \pi\over \tau} \sum_{j=m_1+1}^{m}F_{\delta,\xi}(\xi_j) +o(\de^2)\nonumber
\end{eqnarray}
in view of \eqref{fim}.

\medskip \noindent Let us now expand the potential terms in $J_{\lambda_1,\la_2}(W)$, similarly to the proof of \cite[Theorem 3.2]{EF}. By Lemma \ref{ewfxi} for any $j=1,\dots,m_1$ we find that
\begin{eqnarray*}
\int_{B_{r_0}(\xi_j)}V_1 e^W dv_g 
&=&{e^{\al_{\de,\xi}} \over 8\de_j^2}\left[\int_S \chi_j e^{U_j} \rho_j  e^{ -2 F_{\delta,\xi}} dv_g
-8\delta_j^2 \int_{A_{2r_0}(\xi_j)} \frac{\chi_j \rho_j}{|y_{\xi_j}(x)|^4} dv_g+O(\delta^4 |\log \delta|)\right].
\end{eqnarray*}
By Lemma \ref{ieuf} (with $f(x)=e^{\varphi_j}\rho_j e^{\alpha_{\delta,\xi}-2F_{\delta,\xi}}$) we can now deduce that
\begin{eqnarray*}
&& 8 \delta_j^2 e^{-\al_{\de,\xi}}  \int_{B_{r_0}(\xi_j)}V_1 e^W dv_g = 8\pi \rho_j(\xi_j)  e^{ -2F_{\delta,\xi}(\xi_j)}
-4\pi   \left(\Delta_g  \rho_j (\xi_j)-2 K(\xi_j)\rho_j(\xi_j)\right) \delta_j^2 \log \delta_j \\
&&-2 \left(\Delta_g  \rho_j (\xi_j)-2 K(\xi_j)\rho_j(\xi_j)\right) \left( \int_{\mathbb{R}^2} {\chi'(|y|) \log |y|\over |y| } dy +\pi \right) \de_j^2
+4 \de_j^2 \rho_j(\xi_j) \int_{\mathbb{R}^2} {\chi'(|y|) \over
|y|^3 } dy\\
&&+8\de_j^2\int_{B_{r_0}(\xi_j)} \left[ V_1e^{8\pi \sum\limits_{j=1}^m G(x,\xi_j) - {8\pi\over\tau} \sum\limits_{i=m_1+1}^m G(x,\xi_i) } - e^{-\varphi_j} \frac{ P_2(e^{\varphi_j}\rho_j)}{|y_{\xi_j}(x)|^4}\right]dv_g\\
&&-8\de_j^2\int_{A_{2r_0}(\xi_j)}  \chi_j e^{-\varphi_j} \frac{ P_2(e^{\varphi_j}\rho_j)}{|y_{\xi_j}(x)|^4}\,dv_g+o(\delta^2)
\end{eqnarray*}
in view of $\frac{\rho_j(x)}{|y_{\xi_j}(x)|^4}=V_1e^{8\pi \sum\limits_{j=1}^{m_1} G(x,\xi_j) - {8\pi\over\tau} \sum\limits_{i=m_1+1}^m G(x,\xi_i)}$ in $B_{r_0}(\xi_j)$ and by (\ref{equationvarphi})
\begin{eqnarray} \label{gaussian}
\Delta_g  \left[e^{\varphi_j} \rho_j \right](\xi_j)=
\Delta_g  \rho_j (\xi_j)-2 K(\xi_j)\rho_j(\xi_j).
\end{eqnarray}
Now, by Lemma \ref{ewfxi} for  any $j=m_1+1,\dots,m$ we  find that
\begin{eqnarray*}
\int_{B_{r_0}(\xi_j)}V_1 e^W dv_g &=&\int_{B_{r_0} (\xi_j)} V_1\Big[ {\rho_j\over V_2} \Big]^{-1/\tau} e^{-{1\over\tau}[U_j-\log(8\de_j^2)] +\al_{\de,\xi} + O(\de^2)} dv_g\\
&=& e^{\al_{\de,\xi}} \bigg[ \int_{B_{r_0} (\xi_j)} V_1 e^{ 8\pi \sum\limits_{j=1}^{m_1}
G(x,\xi_j) - {8\pi\over\tau} \sum\limits_{i=m_1+1}^m G(x,\xi_i) } dv_g + O(\de^2 )\bigg].
\end{eqnarray*}
On the other hand, we have that
\begin{equation*}
\int_{S \sm \cup_{j=1}^m B_{r_0}(\xi_j)}
V_1 e^W dv_g= e^{\al_{\de,\xi}} \bigg[\int_{S \sm \cup_{j=1}^m B_{r_0}(\xi_j)} V_1 e^{ 8\pi \sum\limits_{j=1}^{m_1} G(x,\xi_j) - {8\pi\over\tau} \sum\limits_{i=m_1+1}^m G(x,\xi_i)}dv_g+O(\de^2 )\bigg].
\end{equation*}
Since
\begin{equation}\label{s1m1em2f}
\sum_{j=1}^{m_1} e^{-2 F_{\delta,\xi}(\xi_j)}=m_1-2 \sum_{j=1}^{m_1}
F_{\delta,\xi}(\xi_j)+O(\delta^4)
\end{equation}
and by (\ref{repla0}) there holds $\ds \delta_j^2 \log \delta_j=\rho_j(\xi_j) \mu_i^2 \delta^2 \log \delta + \rho_j(\xi_j)\mu_i^2\de^2\log\mu_i +\frac{1}{2}\rho_j(\xi_j) \log \rho_j(\xi_j) \mu_i^2 \delta^2,$
we then obtain that
\begin{eqnarray}
\frac{1}{\pi}e^{-\alpha_{\delta,\xi}}\mu_1^2\delta^2 \int_{S}V_1 e^W
dv_g=m_1
-   \frac{A_1^*(\xi)}{8\pi} \mu_1^2\delta^2 \log(\mu_1 \delta)  +\frac{B_{1,\chi}(\xi)}{8\pi} \mu_1^2 \delta^2-2 \sum_{j=1}^{m_1}
F_{\delta,\xi}(\xi_j)+o(\delta^2),\label{intV1eW}
\end{eqnarray}
where
\begin{eqnarray*}
B_{1,\chi}(\xi)&=&-2\pi \sum_{j=1}^{m_1} [\Delta_g  \rho_j(\xi_j) -2 K(\xi_j) \rho_j(\xi_j)] \log \rho_j(\xi_j) \nonumber \\
&&-\frac{A_1^*(\xi) }{2 \pi} \bigg( \int_{\mathbb{R}^2} {\chi'(|y|) \log |y|\over |y| } dy +\pi \bigg) +4 \int_{\mathbb{R}^2} {\chi'(|y|) \over |y|^3 } dy
\sum_{j=1}^{m_1} \rho_j(\xi_j)\\
&& +8 \int_S \left[V_1e^{8\pi \sum\limits_{j=1}^{m_1} G(x,\xi_j)-{8\pi\over \tau}\sum\limits_{l=m_1+1}^{m}G(x,\xi_l)}-\sum_{j=1}^{m_1}\chi_j e^{-\varphi_j} \frac{
P_2(e^{\varphi_j}\rho_j)}{|y_{\xi_j }(x)|^4}\right]
dv_g, \nonumber
\end{eqnarray*}
By integration by parts on integrals involving $\chi$ and the splitting of $S$ as the union of $\cup_{j=1}^{m_1} B_r(\xi_j)$ and $S \setminus \cup_{j=1}^{m_1} B_r(\xi_j)$, $r\leq r_0$, we easily deduce that
\begin{eqnarray*}
B_{1,\chi}(\xi)&=& -2\pi \sum_{j=1}^{m_1} [\Delta_g  \rho_j(\xi_j) -2 K(\xi_j) \rho_j(\xi_j)] \log \rho_j(\xi_j) -\frac{A_1^*(\xi)}{2}\\
&&+ 8 \int_{S \setminus \cup_{j=1}^{m_1} B_r(\xi_j)}   V_1e^{8\pi \sum\limits_{j=1}^m G(x,\xi_j) -{8\pi\over\tau }\sum\limits_{l=m_1+1}^m G(x,\xi_l) } dv_g-\frac{8\pi}{r^2} \sum_{j=1}^{m_1} \rho_j(\xi_j)-A_1^*(\xi) \log \frac{1}{r} \\
&&+8\sum_{j=1}^{m_1} \int_{B_r(\xi_j)}\frac{
e^{\varphi_j(x)}\rho_j(x)-P_2(e^{\varphi_j}\rho_j)(x)}{|y_{\xi_j}(x)|^4}\,e^{-\varphi_j(x)} dv_g
\end{eqnarray*}
in view of \eqref{gaussian} and the definitions of $A_1^*(\xi)$, $P_2(e^{\varphi_j}\rho_j)$. As a by-product we have that $B_{1,\chi}(\xi)$ does not depend on $\chi$ and $r \leq r_0$. Since
$$\lim_{r \to 0} \int_{B_r(\xi_j)}\frac{
e^{\varphi_j(x)}\rho_j(x)-P_2(e^{\varphi_j}\rho_j)(x)}{|y_{\xi_j}(x)|^4}\,e^{-\varphi_j(x)} dv_g=0$$
in view of $e^{\varphi_j(x)}\rho_j(x)-P_2(e^{\varphi_j}\rho_j)(x)=o(|y_{\xi_j}(x)|^2)$ as $x \to \xi_j$, we have that $B_{1,\chi}(\xi)$ coincides with $B_1^*(\xi)$ as defined in \eqref{Bk} with $k=1$.

\medskip\noindent Similar as above, by Lemmata \ref{ewfxi}, \ref{ieuf} (with $f(x)=e^{\varphi_j}\rho_j e^{-\tau \alpha_{\delta,\xi}+2\tau F_{\delta,\xi}}$),  \eqref{gaussian}, 
\begin{equation}\label{sm11me2tf}
\sum_{j=m_1+1}^{m} e^{2\tau F_{\delta,\xi}(\xi_j)}=m_2 + 2\tau \sum_{j=m_1+1}^{m} F_{\delta,\xi}(\xi_j)+O(\delta^4)
\end{equation}
and by (\ref{repla0}), we then obtain that
\begin{eqnarray}\label{intV2etW}
\frac{1}{\pi}e^{\tau \alpha_{\delta,\xi}}\mu_2^2\delta^2 \int_{S}V_2 e^{-\tau W}
dv_g &= &m_2 -   \frac{A_2^*(\xi)}{8\pi} \mu_2^2\delta^2 \log(\mu_2 \delta)  +\frac{B_{2,\chi}(\xi)}{8\pi} \mu_2^2 \delta^2 \\
&&+ 2\tau \sum_{j=m_1+1}^{m} F_{\delta,\xi}(\xi_j)+o(\delta^2)\nonumber,
\end{eqnarray}
where
\begin{eqnarray*}
B_{2,\chi}(\xi)&=& -2\pi \sum_{j=m_1+1}^m [\Delta_g  \rho_j(\xi_j) -2 K(\xi_j) \rho_j(\xi_j)] \log \rho_j(\xi_j) -\frac{A_2^*(\xi)}{2}\\
&&+ 8 \int_{S \setminus \cup_{j=m_1+1}^m B_r(\xi_j)}   V_2e^{-8\pi\tau  \sum\limits_{j=1}^m G(x,\xi_j) + 8\pi\sum\limits_{l=m_1+1}^m G(x,\xi_l)} dv_g-\frac{8\pi}{r^2} \sum_{j=m_1+1}^m \rho_j(\xi_j)\\
&&-A_2^*(\xi) \log \frac{1}{r} +8\sum_{j=m_1+1}^m \int_{B_r(\xi_j)}\frac{
e^{\varphi_j(x)}\rho_j(x)-P_2(e^{\varphi_j}\rho_j)(x)}{|y_{\xi_j}(x)|^4}\,e^{-\varphi_j(x)} dv_g,
\end{eqnarray*}
$B_{2,\chi}(\xi)$ does not depend on $\chi$ and $r \leq r_0$, and coincides with $B_2^*(\xi)$ as defined in (\ref{Bk}) with $k=2$.

\medskip \noindent Finally, from \eqref{repla1}, expansions \eqref{gtJW}, \eqref{intV1eW} and \eqref{intV2etW} and Taylor's expansion for $a\ge 1$, $\log(a+t)=\log a + {t\over a} + O(t^2)$ as $t\to 0$, we get the expansion \eqref{JUt} as $\delta \to 0$ and the proof is complete.
\end{proof}

\medskip \noindent We establish now expansion \eqref{JUt} in a $C^1$-sense in $\xi$, where the derivatives in $\xi$  are with respect to a given coordinate system. Recall we use ideas in \cite[Theorem 3.2]{EF}.
\begin{proof}[{\bf Proof (of (\ref{JUt}) in $C^1(\Xi)$):}]
We just need to expand the derivatives of $J_{\lambda_1,\la_2}(W)$ in $\xi$. Let us fix $i\in\{1,2\}$ and $j\in\{1,\dots,m\}$. We have that
$$\fr_{(\xi_j)_i}[J_{\la_1,\la_2}(W)]=-\int_S\lf[\lab W+{\la_1 V_1e^{W}\over \int_S V_1e^W dv_g} -{\la_1\tau V_2e^{-\tau W}\over
\int_S V_2e^{-\tau W} dv_g}\rg]\fr_{(\xi_j)_i}W dv_g.$$ 
Notice that as in Lemma \ref{ewfxi}, it follows that
\begin{eqnarray}\label{dxiw}
\fr_{(\xi_j)_i}W_q&=&-2{\chi_q \over
\de_q^2+|y_{\xi_q}(x)|^2}\left[\fr_{(\xi_j)_i} |y_{\xi_q}(x)|^2+\delta_q^2
\fr_{(\xi_j)_i} (\log \rho_q(\xi_q))
\right]\\
&&-4 \log |y_{\xi_q}(x)|\fr_{(\xi_j)_i}\chi_q +8\pi
\fr_{(\xi_j)_i} H(x,\xi_q)+O(\de^2|\log\de|)\nonumber
\end{eqnarray}
does hold uniformly in $S$. Hence, by using \eqref{dxiw} and expansions in the proof of (35) in $C^1(\Xi)$ in \cite[Theorem 3.2]{EF}, we deduce that
\begin{eqnarray}
-\int_S \lab W\fr_{(\xi_j)_i}W dv_g &=& \sum_{l=1}^{m_1} \int_S
\chi_l e^{-\varphi_l} e^{U_l}\fr_{(\xi_j)_i}W dv_g - {1\over \tau} \sum_{l=m_1+1}^m \int_S
\chi_l e^{-\varphi_l} e^{U_l}\fr_{(\xi_j)_i}W dv_g \nonumber \\
&=& -32 \pi^2 \fr_{(\xi_j)_i} \varphi^*_m(\xi) +O(\de^2|\log\de|)
\label{1term}
\end{eqnarray}
for $j\in\{1,\dots,m_1\}$. Similarly, for $j\in\{m_1+1,\dots,m\}$ we compute
$$-\int_S \lab W\fr_{(\xi_j)_i}W dv_g=-32 \pi^2 \fr_{(\xi_j)_i} \varphi_m(\xi) +O(\de^2|\log\de|). $$
In order to give an expansion of the second term in $\fr_{(\xi_j)_i}[J_\la(W)]$, first observe that by Lemma \ref{ewfxi} there hold
\begin{equation}\label{v1ewxij}
V_1e^W={e^{\alpha_{\de,\xi}-2 F_{\delta,\xi}(x)}\over 8\de_j^2} \rho_je^{U_j}[1+O(\de^4|\log \delta|)]\quad\text{uniformly in $B_{r_0}(\xi_j)$, $j=1\dots,m_1$}
\end{equation}
\begin{equation}\label{v1ewsmxij}
\text{$V_1e^W=O(1)$ uniformly in $S \setminus \cup_{j=1}^{m_1} B_{r_0}(\xi_j)$,}
\end{equation}
\begin{equation}\label{v2emtwxij}
V_2e^{-\tau W}={e^{-\tau \alpha_{\de,\xi} + 2\tau F_{\delta,\xi}(x)}\over 8\de_j^2} \rho_je^{U_j}[1+O(\de^4|\log \delta|)]\quad\text{uniformly in $B_{r_0}(\xi_j)$, $j=m_1+1\dots,m$}
\end{equation}
\begin{equation}\label{v2emtwsmxij}
\text{ and $V_2e^{-\tau W}=O(1)$ uniformly in $S \setminus \cup_{j=m_1+1}^{m} B_{r_0}(\xi_j)$}.
\end{equation}
So, arguing in the same way as in the proof of (35) in $C^1(\Xi)$ in \cite[Theorem 3.2]{EF} and taking into account that for $k=1,2$
$$\int_S V_ke^{(-\tau)^{k-1} W} \fr_{(\xi_j)_i}W dv_g= \sum_{l = 1}^{m_1} \int_SV_ke^{(-\tau)^{k-1} W}\fr_{(\xi_j)_i}W_l-{1\over \tau}\sum_{l=m_1+1}^m\int_SV_ke^{(-\tau)^{k-1} W} \fr_{(\xi_j)_i}W_l$$
we have that 
\begin{equation}
\int_S {V_ke^{(-\tau)^{k-1} W}\over \int_S
V_ke^{(-\tau)^{k-1} W} dv_g} \fr_{(\xi_j)_i}W dv_g=O(\de^2|\log\de|),\quad k=1,2.\label{2term}
\end{equation}
In conclusion, by (\ref{1term})-(\ref{2term}) we can write
\begin{eqnarray*} 
\fr_{(\xi_j)_i}[J_{\la_1,\la_2}(W)]= -32 \pi^2 \fr_{(\xi_j)_i} \varphi^*_m(\xi)+O(\delta^2 |\log \delta|).
\end{eqnarray*}
and the proof is complete.
\end{proof}

\medskip \noindent Finally, we address the expansions for the derivatives of $J_{\la_1,\la_2}(W)$ in $\mu$. Recall that we argue similar to the proof of (35) in $C^2(\mathbb{R})$ in \cite[Theorem 3.2]{EF}.
\begin{proof}[{\bf Proof (of (\ref{JUt}) in $C^2(\mathbb{R}^2)$):}]
We just focus on the first and second derivative of $J_{\lambda_1,\la_2}(W)$ in $\mu_i$, $i=1,2$. Since $\fr_{\mu_i}= \de \rho_l^{\frac{1}{2}}(\xi_l) \fr_{\de_l} $, $i=1$ for $l\in\{1,\dots,m_1\}$ and $i=2$ for $l\in\{m_1+1,\dots,m\}$, in view of \eqref{repla0}, arguing as in Lemma \ref{ewfxi}, it is easy to show that
\begin{eqnarray}\label{ddw}
&&\de^{-1}\rho_l^{-\frac{1}{2}}(\xi_l)  \fr_{\mu_i} W_l=- \chi_l \frac{4
\delta_l}{\delta_l^2+|y_{\xi_l}(x)|^2}+ \beta_{\delta_l,\xi_l}-4
\delta_l F_{\xi_l}+O(\delta^3 |\log
\delta|)\\
&&\de^{-2}\rho_l^{-1}(\xi_l) \fr_{\mu_i\mu_i} W_l=4\chi_l
\frac{\delta_l^2-|y_{\xi_l}(x)|^2}{(\delta_l^2+|y_{\xi_l}(x)|^2)^2}+
\gamma_{\delta_l,\xi_l}-4 F_{\xi_l}+O(\delta^2 |\log
\delta|)\label{dddw}
\end{eqnarray}
do hold uniformly in $S$, where
$$\beta_{\delta_l,\xi_l}=-{8\pi\over|S|} \delta_l \log \delta_l+{4\delta_l \over|S|}\lf(\int_{\mathbb{R}^2}
\chi(|y|) \frac{e^{\hat \varphi_\xi(y)}-1}{|y|^2}dy- \int_{\mathbb{R}^2} {\chi'(|y|) \log |y|\over
|y| } dy \rg)$$
and
$$\gamma_{\delta_l,\xi_l}=-{8\pi\over|S|} \log \delta_l+{4 \over|S|}\lf(\int_{\mathbb{R}^2} \chi(|y|) \frac{e^{\hat \varphi_\xi(y)}-1}{|y|^2}dy-2\pi -\int_{\mathbb{R}^2} {\chi'(|y|) \log |y|\over |y| } dy \rg).$$ 
Note that $\fr_{\mu_i}W_l=0$ either if $i=1$ and $l\in\{m_1+1,\dots,m\}$ or $i=2$ and $l\in \{1,\dots,m\}$. Let us stress that $\fr_{\mu_i\mu_k}W_l=0$ for all $l=1,\dots,m$ and $i\ne k$, so that $\fr_{\mu_i\mu_k}W=0$ for $i\ne k$. By Lemma \ref{ieuf} we then have that either for $i=1$, $l\in\{1,\dots,m_1\}$ or $i=2$, $l\in\{m_1+1,\dots,m\}$
$$\de^{-1}\rho_l^{-\frac{1}{2}}(\xi_l) \int_S \chi_j e^{-\varphi_j}e^{U_j} \fr_{\mu_i} W_l dv_g= - \frac{16 \pi}{\delta_j} \delta_{jl}  +8\pi
\beta_{\delta_l,\xi_l}-32 \pi  \delta_l F_{\xi_l}(\xi_j)+O(\delta^3|\log \delta|^2),$$
\begin{eqnarray}\label{ieuddmuw}
\de^{-2}\rho_l^{-1}(\xi_l)\int_S \chi_j e^{-\varphi_j}e^{U_j} \fr_{\mu_i\mu_i} W_l dv_g 
 = \frac{16 \pi}{3\delta_j^2} \delta_{jl}  +8\pi
 \gamma_{\delta_l,\xi_l}-32 \pi F_{\xi_l}(\xi_j)+O(\delta^2 |\log \delta|^2)
\end{eqnarray}
and either for $k=1$, $j\in\{1,\dots,m_1\}$ or $k=2$, $l\in\{m_1+1,\dots,m\}$
\begin{eqnarray}\label{ieududw}
&&\de^{-1} \rho_l^{-\frac{1}{2}}(\xi_l) \int_S \chi_j e^{-\varphi_j}e^{U_j} \fr_{\mu_k} U_j \fr_{\mu_i} W_l
dv_g= \frac{2}{\mu_k}\de^{-1} \rho_l^{-\frac{1}{2}}(\xi_l) \int_S \chi_j e^{-\varphi_j}e^{U_j}
\frac{|y_{\xi_j}(x)|^2-\de_j^2}{\de_j^2+|y_{\xi_j}(x)|^2} \fr_{\mu_i}
W_l dv_g\nonumber\\
&&=\frac{32 \pi}{3 \de_j^2} \de\rho_j(\xi_j)^{\frac{1}{2}}
\de_{jl}+O(\de^\gamma)
\end{eqnarray}
in view of $\int_{\mathbb{R}^2} \frac{|y|^2-1}{(1+|y|^2)^3}dy=0$, where $\delta_{jl}$ denotes the
Kronecker's symbol. Note that $\fr_{\mu_k}U_j=0$ either for $k=1$ and $j\in\{m_1+1,\dots,m\}$ or $k=2$ and $j\in\{1,\dots,m_1\}$. Since $\int_S \fr_{\mu_i} W dv_g=\int_S \fr_{\mu_i\mu_k} W dv_g=0$, we then deduce the following expansions:
\begin{eqnarray}
\int_S (-\Delta_g W)  \fr_{\mu_1} W dv_g&=& \sum_{j,l=1}^{m_1} \int_S \chi_j e^{-\varphi_j}e^{U_j} \fr_{\mu_1} W_l dv_g -{1\over\tau} \sum_{j=m_1+1}^m\sum_{l=1}^{m_1} \int_S \chi_j e^{-\varphi_j}e^{U_j} \fr_{\mu_1} W_l dv_g \label{deW} \nonumber\\
&=&- \frac{16 \pi m_1}{\mu_1}
+8\pi m_1\de \sum_{l=1}^{m_1} \rho_l^{\frac{1}{2}}(\xi_l) \beta_{\delta_l,\xi_l}-32 \pi \mu_1\de^2 \sum_{j,l=1}^{m_1}
\rho_l(\xi_l) F_{\xi_l}(\xi_j)\\
&& - {8\pi m_2\over \tau}\de \sum_{l=1}^{m_1} \rho_l^{\frac{1}{2}}(\xi_l) \beta_{\delta_l,\xi_l} + {32 \pi\over \tau} \mu_1\de^2 \sum_{j=m_1+1}^{m}\sum_{l=1}^{m_1}
\rho_l(\xi_l) F_{\xi_l}(\xi_j)+O(\delta^4|\log \de|^2), \nonumber
\end{eqnarray}
and
\begin{eqnarray}
\int_S (-\Delta_g W)  \fr_{\mu_2} W dv_g&=& -{1\over \tau} \sum_{j=1}^{m_1}\sum_{l=m_1+1}^m \int_S \chi_j e^{-\varphi_j}e^{U_j} \fr_{\mu_2} W_l dv_g + {1\over\tau^2} \sum_{j,l=m_1+1}^m \int_S \chi_j e^{-\varphi_j}e^{U_j} \fr_{\mu_2} W_l dv_g \label{de2W} \nonumber\\
&=&- \frac{16 \pi m_2}{\mu_2\tau^2} - {8\pi m_1\over \tau}\de \sum_{l=1}^{m_1} \rho_l^{\frac{1}{2}}(\xi_l) \beta_{\delta_l,\xi_l} + {32 \pi \over \tau} \mu_2\de^2 \sum_{j=1}^{m_1}\sum_{l=m_1+1}^{m}\rho_l(\xi_l) F_{\xi_l}(\xi_j)\\
&& + {8\pi m_2\over \tau^2}\de \sum_{l=m_1+1}^{m} \rho_l^{\frac{1}{2}}(\xi_l) \beta_{\delta_l,\xi_l} - {32 \pi\over \tau^2} \mu_2\de^2 \sum_{j,l=m_1+1}^{m} \rho_l(\xi_l) F_{\xi_l}(\xi_j)+O(\delta^4|\log \de|^2), \nonumber
\end{eqnarray}
as $\de \to 0$. Since by Lemma \ref{ewfxi}
there hold \eqref{v1ewxij}, \eqref{v1ewsmxij}
and $\fr_{\mu_1} W= O(\delta^2 |\log \delta|)$ uniformly in $S \setminus \cup_{j=1}^{m_1} B_{r_0}(\xi_j)$, by Lemma \ref{ieuf} we can write that
\begin{eqnarray*}
&\hspace{-0.2cm}\ds\int_S &\hspace{-0.3cm}V_1e^W \fr_{\mu_1} W dv_g = \sum_{j,l=1}^{m_1}
\int_{B_{r_0}(\xi_j)} V_1e^W \fr_{\mu_1} W_l dv_g +O(\delta^2 |\log
\delta|) \\
&=&- \sum_{j=1}^{m_1} {e^{\alpha_{\de,\xi}}\over 2 \mu_1}
\int_{B_{r_0}(\xi_j)} e^{-2 F_{\delta,\xi}(x)}
 \frac{ \rho_je^{U_j}}{\delta_j^2+|y_{\xi_j}(x)|^2}
dv_g\\
&&+\;\pi {e^{\alpha_{\de,\xi}}\over \mu_1^2\de}  \left(m_1 \sum_{l=1}^m \rho_l^{\frac{1}{2}}(\xi_l)
\beta_{\delta_l,\xi_l}-4 \sum_{j,l=1}^{m_1} \rho_l^{\frac{1}{2}}(\xi_l) \delta_l F_{\xi_l}(\xi_j) \right)+\,O(\de|\log \delta|)\\
&=&\pi {e^{\alpha_{\de,\xi}}\over \mu_1^2\de^2} \left(-\frac{2m_1}{\mu_1}+m_1\de \sum_{l=1}^{m_1}
\rho_l^{\frac{1}{2}}(\xi_l) \beta_{\delta_l,\xi_l}-{\mu_1\de^2 \over 8\pi} A_1^*(\xi)- {4\over \mu_1\tau}\sum_{j=1}^{m_1} F_{2,\de,\xi}(\xi_j) +O(\delta^{2+\gamma})\right)
\end{eqnarray*}
in view of (\ref{gaussian}) and from \eqref{s1m1em2f}
\begin{eqnarray*} 
\sum_{j=1}^{m_1} e^{-2F_{\de,\xi}(\xi_j)} 
= m_1 -2 \sum_{j,l=1}^{m_1} \de_l^2  F_{\xi_l}(\xi_j)+{2\over \tau}\sum_{j=1}^{m_1}F_{2,\de,\xi}(\xi_j)+O(\de^4).
\end{eqnarray*}
Combining with (\ref{intV1eW}) we then get that
\begin{eqnarray} \label{Merry1}
\frac{\int_S V_1e^W \fr_{\mu_1} W dv_g}{\int_S V_1e^W dv_g} &=& -\frac{2}{\mu_1}+\de \sum_{l=1}^{m_1}
\rho_l^{\frac{1}{2}}(\xi_l) \beta_{\delta_l,\xi_l}-{\de^2A_1^*(\xi) \over 8\pi m_1} [\mu_1+2\mu_1\log\mu_1]\\
&&-\frac{A_1^*(\xi)}{4 \pi m_1} \mu_1\delta^2 \log \delta+\frac{B_1^*(\xi)}{4\pi m_1} \mu_1\delta^2-\frac{4}{m_1 \mu_1}\sum_{j=1}^{m_1} F_{1,\delta,\xi}(\xi_j)+o(\de^2).\nonumber
\end{eqnarray}
Similarly as above, there hold \eqref{v2emtwxij}, \eqref{v2emtwsmxij} and
and $\fr_{\mu_1} W= O(\delta^2 |\log \delta|)$ uniformly in $S\sm\cup_{j=1}^{m_1}B_{r_0}(\xi_j)$, so that 
\begin{eqnarray*}
&&\int_S V_2e^{-\tau W} \fr_{\mu_1} W dv_g = \sum_{j,l=1}^{m_1}
\int_{B_{r_0}(\xi_j)} V_2e^{-\tau W} \fr_{\mu_1} W_l dv_g+ \sum_{j=m_1+1}^{m}\sum_{l=1}^{m_1}
\int_{B_{r_0}(\xi_j)} V_2e^{-\tau W} \fr_{\mu_1} W_l dv_g\\
&& \ \ + \: O(\delta^2 |\log
\delta|)\\
&&=\pi {e^{-\tau\alpha_{\de,\xi}}\over \mu_2^2\de^2} \left(m_2\de \sum_{l=1}^{m_1} \rho_l^{\frac{1}{2}}(\xi_l) \beta_{\delta_l,\xi_l} - {4\over \mu_1}\sum_{j=m_1+1}^{m} F_{1,\de,\xi}(\xi_j) +O(\delta^{4}|\log\de|)\right)
\end{eqnarray*}
in view of $\tau>0$, \eqref{sm11me2tf} and
\begin{eqnarray*}
\int_{B_{r_0}(\xi_j)} 
 \frac{ V_2e^{-\tau W}}{\delta_j^2+|y_{\xi_j}(x)|^2}
dv_g=O\lf(\int_{B_{r_0}(\xi_j)} (\de_j^2+|y_{\xi_j}(x)|^2)^{\tau-1}dv_g\rg)=O(1).
\end{eqnarray*}
Combining with (\ref{intV2etW}) we then get that
\begin{eqnarray} \label{Merry0}
\frac{\int_S V_2e^{-\tau W} \fr_{\mu_1} W dv_g}{\int_S V_2e^{-\tau W} dv_g} &=& \de \sum_{l=1}^{m_1}
\rho_l^{\frac{1}{2}}(\xi_l) \beta_{\delta_l,\xi_l} - \frac{4}{m_2 \mu_1}\sum_{j=1}^{m_1} F_{1,\delta,\xi}(\xi_j)+o(\de^2),
\end{eqnarray}
which yields to
\begin{eqnarray} \label{derivmu1}
\fr_{\mu_1}[J_{\la_1,\la_2}(W)]&=&
\int_S (-\Delta_g W)  \fr_{\mu_1} W dv_g-\la_1 \frac{\int_S V_1e^W \fr_{\mu_1} W dv_g}{\int_S V_1e^W dv_g} +\la_2\tau \frac{\int_S V_2e^{-\tau W} \fr_{\mu_1} W dv_g}{\int_S V_2e^{-\tau W} dv_g} \nonumber \\
&=& \frac{2(\lambda_1-8\pi m_1)}{\mu_1}+2 A_1^*(\xi) \mu_1 \delta^2 \log \delta+[A_1^*(\xi)\{\mu_1+2\mu_1\log\mu_1\}-2 B_1^*(\xi)\mu_1] \delta^2\nonumber \\
&&+o(\de^2)
\end{eqnarray}
in view of \eqref{deW}, so that we deduce the validity of (\ref{JUt}) for the first derivative in $\mu_1$. Now, for the first derivative in $\mu_2$, similarly as above we have that
\begin{eqnarray} \label{Merry}
\frac{\int_S V_1e^{W} \fr_{\mu_2} W dv_g}{\int_S V_1e^{ W} dv_g} &=& -{\de\over\tau} \sum_{l=m_1+1}^{m}
\rho_l^{\frac{1}{2}}(\xi_l) \beta_{\delta_l,\xi_l} + \frac{4 }{m_1 \mu_2\tau}\sum_{j=1}^{m_1} F_{2,\delta,\xi}(\xi_j)+o(\de^2).
\end{eqnarray}
in view of \eqref{intV1eW}, and 
\begin{eqnarray} \label{Merrymu2}
\frac{\int_S V_2e^{-\tau W} \fr_{\mu_2} W dv_g}{\int_S V_2e^{-\tau W} dv_g} &=& \frac{2}{\mu_2\tau} - {\de\over\tau} \sum_{l=m_1+1}^{m} \rho_l^{\frac{1}{2}}(\xi_l) \beta_{\delta_l,\xi_l}+{\de^2A_2^*(\xi) \over 8\pi m_2\tau} [\mu_2+2\mu_2\log\mu_2]\\
&&+\frac{A_2^*(\xi)}{4 \pi m_2\tau} \mu_2\delta^2 \log \delta - \frac{B_2^*(\xi)}{4\pi m_2\tau} \mu_2\delta^2 + \frac{4}{m_2 \mu_2\tau}\sum_{j=m_1+1}^{m} F_{2,\delta,\xi}(\xi_j)+o(\de^2).\nonumber
\end{eqnarray}
by using \eqref{sm11me2tf} and combining with (\ref{intV2etW}). Thus, by using \eqref{de2W} we conclude the validity of (\ref{JUt}) for the first derivative in $\mu_2$:
\begin{eqnarray} \label{derivmu2}
\fr_{\mu_2}[J_{\la_1,\la_2}(W)] &=& \frac{2(\lambda_2\tau^2-8\pi m_2)}{\mu_2\tau^2}+\frac{2 A_2^*(\xi)}{\tau^2} \mu_2 \delta^2 \log \delta+[A_2^*(\xi)\{\mu_2+2\mu_2\log\mu_2\}-2 B_2^*(\xi)\mu_2] \frac{\delta^2}{\tau^2} \nonumber \\
&&+o(\de^2)
\end{eqnarray}

\medskip \noindent Towards the expansion of the second derivatives in $\mu$, we proceed in a similar way to obtain \eqref{derivmu1} and \eqref{derivmu2} with the aid of the expansions \eqref{ddw} for $\fr_{\mu_i}W$ and \eqref{dddw} for $\fr_{\mu_i\mu_i} W_l$, \eqref{ieuddmuw} and \eqref{ieududw} (see also validity of expansion (35) in $C^2(\mathbb{R})$ in \cite[Theorem 3.2]{EF}). We omit the details, thus, we conclude the validity of (\ref{JUt}) also for the second derivatives in $\mu$ and the proof is complete.
\end{proof}


\section{Proof of Theorem \ref{main3}}\label{pthm3}

In this section, we shall study the existence of blowing-up solutions as $\la_1\to 8\pi m_1$ and $\la_2\tau^2\to 0$, which resembles the equation \eqref{mfe}. For simplicity, we shall denote $m_1=m$ so that our approximating solution is $\ds W(x)=\sum_{j=1}^m W_j(x)$, and we look for solutions to \eqref{mfewt} in the form $u=W+\phi$. Assumptions \eqref{repla0}-\eqref{repla1} are replaced by
\begin{equation}\label{replam20}
\begin{split}
\de_j^2=\mu^2\de^2\rho_j(\xi_j),\ \ j=1,\dots,m\qquad &\text{with}\qquad 0<\mu\le C_0 \\[0.1cm]
|\lambda_1-8\pi m |\le C \de^2|\log
\de| \qquad&\text{and}\qquad 0<\lambda_2\tau^2 \le C \de^2|\log \de|.
\end{split}
\end{equation}
Notice from similar computations above to obtain \eqref{intV2etW}, we have that
$$\int_S V_2e^{-\tau W} dv_g= e^{-\tau \al_{\de,\xi}}\lf[\int_S V_2 e^{-8\pi \tau \sum\limits_{j=1}^m G(x,\xi_j)}dv_g +O(\de^2)\rg]\ge \eta_0>0$$
for some $\eta_0>0$. By conditions \eqref{replam20} we get that
\begin{equation}\label{la2v2}
{\la_2\tau V_2 e^{-\tau W}\over \int_S V_2e^{-\tau W}} =O(\de^2|\log\de|)\quad \text{uniformly in $S$}.
\end{equation}
Hence, it follows estimate \eqref{re}. Now, denote $\ds Z=\sum_{l=1}^m Z_{0l}$ and $PZ$ its projection according to \eqref{ePZ}. By using \eqref{la2v2} and similar arguments used in the proofs of \cite[Proposition 4.1]{EF} and Proposition \ref{p2}, it follows the invertibility of $L$ in \eqref{ol} in this case (as $\la_1\to 8\pi m	$ and $\la_2\tau^2\to 0$), and we deduced the following fact. 
\begin{prop}\label{lpnlabis}
There exists $\delta_0>0$ so that for
all $0<\delta\leq \delta_0$, $\mu\in (0,C_0]$, $\xi \in \Xi$ problem
\begin{equation*} 
\left\{ \begin{array}{ll}
L(\phi)= -[R+N(\phi)] +c_0\Delta_g PZ +\displaystyle \sum_{i=1}^{2}\sum_{j=1}^m c_{ij} \Delta_g
PZ_{ij}& \text{in } S\\
\int_S \phi \Delta_g PZ dv_g=\int_S \phi \Delta_g PZ_{ij} dv_g= 0 &\forall\:
i=1,2,\, j=1,\dots,m
\end{array} \right.
\end{equation*}
admits a unique solution $\phi(\mu,\xi) \in \bar H \cap
W^{2,2}(S)$ and $c_0(\mu,\xi),\,c_{ij}(\mu,\xi) \in \R$, $i=1,2$
and $j=1,\dots,m$, where $\de_j>0$ are as in \eqref{replam20} and
$N$, $R$ are given by \eqref{nlt}, \eqref{R}, respectively. Moreover, the map
$(\mu,\xi)\mapsto (\phi(\mu,\xi),c_0(\mu,\xi),c_{ij}(\mu,\xi))$ is
twice-differentiable in $\mu$ and one-differentiable in $\xi$ with
$$\|\phi\|_\infty+{\|\fr_\mu \phi \|_\infty\over |\log\de|}+\sum_{i=1}^2 \sum_{j=1}^m{\de \|\fr_{(\xi_j)_i} \phi \|_\infty\over |\log\de|} + {\|\fr_{\mu\mu}\phi\|_\infty\over |\log\de|^2} \le C\left( \delta |\log \delta |  |\nabla \varphi_m^*(\xi)|_g+ \delta^{2-\sigma}|\log\delta|^2\right)$$
\end{prop}
\noindent As in the case $m_2\ge 1$, the function $[W+\phi](\mu,\xi)$ will be a true solution of \eqref{ephi} if $\mu\in[C_0^{-1},C_0]$ and $\xi\in \Xi$ are such that
$c_{0}(\mu,\xi)=c_{ij}(\mu, \xi)=0$ for all $i=1,2,$ and $j=1,\dots,m$. Similarly to Lemma \ref{cpfc0bis}, this problem is equivalent to finding critical
points of the reduced energy $E_{\lambda_1,\la_2}(\mu, \xi)= J_{\lambda_1,\la_2}\big([W+\phi](\mu,\xi)\big)$, where $J_{\lambda_1,\la_2}$ is given by \eqref{energy}. Notice that
$$\la_2\log\lf(\int_S V_2e^{-\tau W} dv_g\rg)=-\la_2\tau \al_{\de,\xi} + \la_2\log\lf(\int_S V_2 e^{-8\pi \tau \sum\limits_{j=1}^m G(x,\xi_j)}dv_g\rg) + O(\de^4|\log\de|).$$
Let us stress that $\ds\la_2\log\bigg(\int_S V_2 e^{-8\pi \tau \sum\limits_{j=1}^m G(x,\xi_j)}dv_g\bigg) $ is independent of $\mu$. Taking into account computations in the proof of \cite[Theorem 3.2]{EF} and similar ones in the proof of Theorem \ref{expansionenergy} we have that
\begin{equation*}
\begin{split}
J_{\la_1,\la_2}(W)=&\,-8\pi m -\la_1\log(\pi m) +2(\la_1-8\pi m)\log(\mu\de) -32\pi^2\varphi_m^*(\xi) +A(\xi)\mu^2\de^2\log\de\\
& + [A(\xi)\mu^2\log\mu - B(\xi) \mu^2]\de^2-\la_2\log\lf(\int_S V_2 e^{-8\pi \tau \sum\limits_{j=1}^m G(x,\xi_j)}dv_g\rg)  + o(\de^2).
\end{split}
\end{equation*}
Consequently, from estimates in Appendix B we obtain that
\begin{theo} \label{fullexpansionenergym20}
Assume \eqref{replam20}. The following expansion does
hold
\begin{eqnarray*} 
E_{\lambda_1,\la_2} (\mu,\xi) &=&-8\pi m- \lambda_1 \log (\pi m) -  2\big(\lambda_1 -8\pi m \big)\log\de  -32\pi^2 \varphi_m^*(\xi)\\
&&\,\,+ 2\big(\lambda_1 -8\pi m \big)\log\mu  + A(\xi) \mu^2\delta^2\log  \delta + \lf[A(\xi) \mu^2 \log\mu- B(\xi)\mu^2\rg] \delta^2  \nonumber \\
&&\,\,  -\la_2\log\lf(\int_S V_2 e^{-8\pi \tau \sum\limits_{j=1}^m G(x,\xi_j)}dv_g\rg) + o(\de^2) +r_{\lambda_1,\la_2}(\mu,\xi) \nonumber
\end{eqnarray*}
in $C^2(\mathbb{R})$ and $C^1(\Xi)$ as $\de\to 0^+$, where
$\varphi_m^*(\xi)$, $A(\xi)$ and $B(\xi)$  are given by \eqref{fim}, \eqref{vk} and \eqref{Bk} with $k=1$, respectively. The term $r_{\lambda_1,\la_2}(\mu,\xi)$ satisfies \eqref{rlambda} for some $C>0$ independent of $(\mu,\xi)\in (0,C_0]\times\Xi$.
\end{theo}

\begin{proof}[{\bf Proof (of Theorem \ref{main3}):}] We argue in the same way as in the proof of Theorem \ref{main2} with $k=1$.
\end{proof}

\section{Appendix A}\label{appeA}
\noindent We shall argue in the same way as in Appendix A in \cite{EF}. We first address a-priori estimates for the operator $L$ when all the $c_{ij}$'s vanish:
\begin{prop} \label{p1}
There exists $\delta_0>0$ and $C>0$ so that, for all $0<\delta\leq
\delta_0$, $h\in C(S)$ with $\int_Sh dv_g=0$, $\xi \in \Xi$ and
$\phi \in H_0^1(S) \cap W^{2,2}(S)$ a solution of \eqref{plco}
with $c_{0i}=c_{ij}=0$, $i=1,2$ and $j=1,\dots,m$,
one has
\begin{equation*} 
\|\phi \|_\infty \le C | \log \de | \|h\|_*.
\end{equation*}
\end{prop}
\begin{proof}[\dem] By contradiction, assume the existence of sequences $\de \to 0$, $\mu=(\mu_1,\mu_2)$ with $\mu\to\mu^*$, points $\xi \in \Xi$ with $\xi \to \xi^*$, functions $h$ with $|\log \de| \|h\|_*=o(1)$ and solutions $\phi$ with $\|\phi\|_\infty=1$. Recall that $\de_j^2=\mu_i\de^2 \rho_j(\xi_j)$. Setting $\mathcal{K}_i=\frac{\la_i\tau^{2(i-1)} V_ie^{(-\tau )^{i-1} W}}{\int_S V_ie^{(-\tau)^{i-1} W}dv_g}$, $\psi_i=\phi + \tilde c_i(\phi)$, $\tilde c_i(\phi)=-\frac{\int_S V_ie^{(-\tau)^{i-1} W}\phi dv_g}{\int_S V_ie^{(-\tau)^{i-1} W}dv_g}$ for $i=1,2$, we have that $\psi_1-\tilde c_1(\phi) = \psi_2 - \tilde c_2(\phi)$, $\Delta_g \psi_1+\mathcal{K}_1 \psi_1 + \ml{K}_2[\psi_1-\tilde c_1(\phi)+\tilde c_2(\phi)]=h$ and $\Delta_g \psi_2+\mathcal{K}_1 [\psi_2-\tilde c_2(\phi)+\tilde c_1(\phi)] + \ml{K}_2\psi_2=h$ in $S$ and $\psi_i$, $i=1,2$ does satisfy the same orthogonality conditions as $\phi$.

\medskip \noindent Since $\|\psi_{i,n}\|_\infty\le 2\|\phi_n\|_\infty \le 2$ and $\Delta_g\psi_i =o(1)$ in $C_{\hbox{loc}}(S \setminus \{\xi_1^*,\dots, \xi_m^*\})$, we can assume that $\psi_{i,n} \to \psi_{i,\infty}$ in $C^1_{\hbox{loc}}(S \sm\{\xi_1^*,\dots,\xi_m^*\})$. Since $\psi_{i,\infty}$ is bounded, it extends to an harmonic function in $S$, and then $\psi_{i,\infty}=\tilde c_{i,0}:= -\lim \frac{\int_S V_i e^{(-\tau)^{i-1} W} \phi dv_g}{\int_S V_i e^{(-\tau)^{i-1} W} dv_g}$ in view of ${1\over |S|}\int_S \psi_{i,n} dv_g=\tilde c_{i,n}(\phi)$.

\medskip \noindent The function $\Psi_{i,j} =\psi_i(y_{\xi_j}^{-1}(\delta_j y))$ $i=1$, for $j=1,\dots,m_1$ and $i=2$ for $j=m_1+1,\dots ,m$  satisfy $\lap \Psi_{1,j} + \mathcal{\tilde K}_{1,j} \Psi_{1,j} + \mathcal{\tilde K}_{2,j}[\Psi_{1,j}-\tilde c_1+\tilde c_2]= \tilde h_j$ and $\lap \Psi_{2,j} + \mathcal{\tilde K}_{1,j}[\Psi_{2,j}-\tilde c_2+\tilde c_1] + \mathcal{\tilde K}_{2,j} \Psi_{2,j} = \tilde h_j$ in $B_{2r_0 \over \de_j}(0)$, where $\mathcal{\tilde K}_{i,j}=\de_j^2 \mathcal{K}_i(y_{\xi_j}^{-1}(\de_j y))$ and $\tilde
h_j(y)=\de_j^2 h (y_{\xi_j}^{-1}(\de_j y))$. Since $|\ti h_j| \le C\|h \|_*$,
$$\mathcal{\tilde K}_{1,j}=
\begin{cases}
{8\over (1+|y|^2)^2}(1+O(\delta^2|\log \delta|))&\text{ for $j=1,\dots,m_1$}\\
O(\de_j^2)&\text{ for $j=m_1+1,\dots,m$ }
\end{cases}$$
and
$$\mathcal{\tilde K}_{2,j}=
\begin{cases}
O(\de_j^2)&\text{ for $j=1,\dots,m_1$}\\
{8\over (1+|y|^2)^2}(1+O(\delta^2|\log \delta|))&\text{ for $j=m_1+1,\dots,m$ }
\end{cases}$$ 
uniformly in $B_{\frac{2r_0}{\delta}}(0)$, in view of Lemma \ref{ewfxi}, \eqref{intV1eW} and (\ref{intV2etW}), up to a sub-sequence, by elliptic estimates $\Psi_{i,j} \to \Psi_{j,\infty}$ with $i=1$ if $j=1,\dots,m_1$ and $i=2$ if $j=m_1+1,\dots,m$ in $C^1_{\hbox{loc}}(\mathbb{R}^2)$, where $\Psi_{j,\infty}$ is a bounded solution of $\Delta \Psi_{j,\infty} + {8\over(1+|y|^2)^2}\Psi_{j,\infty}= 0$ of the form $\Psi_{j,\infty}=\displaystyle \sum_{i=0}^2 a_{ij}Y_i$ (see for example \cite{bp}). Since $-\Delta_g PZ_{ij} =\chi_j e^{-\varphi_j} e^{U_j} Z_{ij}-\frac{1}{|S|}\int_S
\chi_j  e^{-\varphi_j} e^{U_j} Z_{ij} dv_g$ in view of (\ref{ePZ}) and $\Delta_g=e^{-\varphi_j} \Delta$ in $B_{2r_0}(\xi_j)$ through $y_{\xi_j}$, we have that
$$0=-\int_S \psi_l \Delta_g PZ_{ij}=32 \int_{\mathbb{R}^2} \Psi_{l,j} \frac{y_i}{(1+|y|^2)^3} dy
-\frac{32}{|S|} \int_{\mathbb{R}^2} \frac{y_i}{(1+|y|^2)^3} dy \int_S \psi_{l,n} +O(\de^3),$$
with $l=1$ if $j=1,\dots,m_1$ and $l=2$ if $j=m_1+1,\dots,m$. Since then $\int_{\mathbb{R}^2} \Psi_{j,\infty} \frac{y_i}{(1+|y|^2)^3} dy=0$, we deduce that $a_{1j}=a_{2j}=0$. By the orthogonality condition $\int_S \phi \Delta_g PZ_1=0$, similarly we deduce that
\begin{eqnarray*}
0&=&-\sum_{j=1}^{m_1}\int_S \psi_1 \Delta_g PZ_{0j}dv_g \\ 
&=&16 \sum_{j=1}^{m_1}\int_{\mathbb{R}^2} \Psi_j \frac{1-|y|^2}{(1+|y|^2)^3} dy
-\frac{16}{|S|} m_1 \int_{\mathbb{R}^2} \frac{1-|y|^2}{(1+|y|^2)^3} dy \int_S \psi_{1,n} +O(\de^2),
\end{eqnarray*}
which implies $\displaystyle \sum_{j=1}^{m_1} a_{0j}=0$ in view of $\int_{\mathbb{R}^2} \frac{1-|y|^2}{(1+|y|^2)^3} dy=0$. By using the same argument, the orthogonality condition  $\int_S \phi \Delta_g PZ_2=0$ implies that $\displaystyle \sum_{j=m_1+1}^{m} a_{0j}=0$. By dominated convergence we have that
\begin{eqnarray*}
&&\int_S G(y,\xi_j) \mathcal{K}_1 \psi_1 dv_g=  -{1\over 2\pi } \log \de \int_{B_{r_0}(\xi_j)}  \mathcal{K}_1 \psi_1 dv_g+\int_{\mathbb{R}^2} \Big[-{1\over 2\pi }\log |y|+H(\xi_j,\xi_j) \Big] \frac{8}{(1+|y|^2)^2} \Psi_{j,\infty} dy\\
&&+ \sum_{i=1 \atop i\not=j}^{m_1} G(\xi_i,\xi_j)  \int_{\mathbb{R}^2} \frac{8}{(1+|y|^2)^2} \Psi_{i,\infty} dy+o(1)=
-{1\over 2\pi } \log \de \int_{B_{r_0}(\xi_j)}  \mathcal{K}_1 \psi_1 dv_g+4 a_{0j}+o(1)
\end{eqnarray*}
in view of $\int_{\mathbb{R}^2} \log |y| \frac{1-|y|^2}{(1+|y|^2)^3} dy=-\frac{\pi}{2}$ and
\begin{equation*}
\begin{split}
\int_S G(y,\xi_j) \mathcal{K}_2 \psi_2 dv_g =&\,\sum_{i=m_1+1}^{m}G(\xi_i,\xi_j) \int_{\mathbb{R}^2} \frac{8}{(1+|y|^2)^2} \Psi_{i,\infty}(y) \, dy \\
& + O\bigg(\de^2 \int_{B_{r_0}(\xi_j)} |G(y,\xi_j)|dv_g\bigg) +o(1) = o(1)
\end{split}
\end{equation*}
for $j=1,\dots,m_1$. In view of $\int_S \ml{K}_l \psi_l=0$, $l=1,2$ and
$$\bigg|\int_S G(y,\xi_j) h dv_g \bigg| \le C |\log\de | \int_S |h| dv_g+\frac{\|h\|_*}{\delta^2}\bigg|\int_{B_\de(\xi_j)} G(y,\xi_j) dv_g\bigg|\leq C' |\log \delta|\|h\|_*=o(1),$$
by the Green's representation formula
\begin{eqnarray*}
\sum_{j=1}^{m_1}  \Psi_j(0)&=& \sum_{j=1}^{m_1} \psi_1 (\xi_j) ={m_1 \over |S|}\int_S \psi_1 dv_g + \sum_{j=1}^{m_1} \int_S G(y,\xi_j) [ \mathcal{K}_1 \psi_1 + \mathcal{K}_2\psi_2- h ] dv_g\\
&=&m_1 \tilde c_{1,0} +4 \sum_{j=1}^{m_1} a_{0j}+o(1)
\end{eqnarray*}
which gives $\displaystyle 2\sum_{j=1}^{m_1} a_{0j}= m_1\tilde c_{1,0} +4 \displaystyle \sum_{j=1}^{m_1} a_{0j}$ as $n \to +\infty$. Since $\displaystyle \sum_{j=1}^{m_1} a_{0j}=0$, we get that $\tilde c_{1,0}=0$. By using a similar argument, we obtain that
$$\int_S G(y,\xi_j) \mathcal{K}_1 \psi_1 dv_g = o(1)\quad\text{for $j=1,\dots,m_1$ and}$$
$$\int_S G(y,\xi_j) \mathcal{K}_2 \psi_2 dv_g = -{1\over 2\pi } \log \de \int_{B_{r_0}(\xi_j)}  \mathcal{K}_2 \psi_2 dv_g+4 a_{0j}+o(1)$$
for $j=m_1+1,\dots,m$, so that, from the Green's representation formula for $\Psi_j(0)$ for $j=m_1+1,\dots,m$ we get that $\tilde c_{2,0}=0$.\\
Following \cite{EGP}, let $P\hat Z_j \in H_0^1(S)$ be s.t. $\Delta_g P \hat Z_j =\chi_j \Delta_g \hat Z_j -\frac{1}{|S|}\int_S
\chi_j \Delta_g \hat Z_j dv_g$ in $S$, where
$$\hat Z_j(x)=\beta_j\Big(\frac{y_{\xi_j}(x)}{\delta_j}\Big)\,,\qquad \beta_j(y)={4\over 3}[2\log \delta_j+\log (1 + |y|^2 )]\frac{1
- |y|^2}{1 + |y|^2} + {8\over 3} \frac{1}{1+ |y|^2}$$
satisfies $e^{\varphi_j}\Delta_g \hat Z_j+e^{U_j} \hat Z_j=e^{U_j}Z_{0j}$ in $B_{2r_0}(\xi_j)$. Since it is easily seen that $ P\hat Z_j= \chi_j \hat Z_j +{16 \pi \over 3}H(\cdot, \xi_j) +O(\de^2 |\log \de |^2)$ uniformly in $S$, we test the equation of $\psi_1$ against $P\hat Z_j$, $j=1,\dots,m_1$ to get:
\begin{eqnarray*}
&&\hspace{-0.7cm}o(1)=\int_S h P\hat Z_j=\int_S \psi_1 \bigg[\chi_j \Delta_g \hat Z_j -\frac{1}{|S|}\int_S
\chi_j \Delta_g \hat Z_j dv_g \bigg]dv_g + \int_S \big[\mathcal{K}_1 \psi_1 + \ml{K}_2(\psi_1-\tilde c_1+\tilde c_2)\big]  P\hat Z_j dv_g\\
&&= \int_S  \chi_j \psi_1 [ \Delta_g \hat Z_j + \mathcal{K}_1 \hat Z_j]dv_g+o(1)=\int_S \chi_j \psi e^{U_j}Z_{0j} dv_g+o(1) =16 \int_{\mathbb{R}^2} \Psi_j \frac{1-|y|^2}{(1+|y|^2)^3}dy+o(1)
\end{eqnarray*}
in view of , $\int_S \mathcal{K}_1 \psi_1 dv_g=0$, $\int_S \mathcal{K}_2 [\psi_1-\tilde c_1+\tilde c_2]P\hat Z_j dv_g=o(1)$, $\int_S \psi_1 dv_g=o(1)$, $\int_S \chi_j \Delta_g \hat Z_j dv_g=O(1)$, $\int_S \chi_j \psi_1 [\mathcal{K}_1-e^{U_j}]\hat Z_j dv_g=O(\delta^2 |\log \delta|^2)$ and $\int_S h P\hat Z_j=O(|\log \delta|\|h\|_*)=o(1)$, $j=1,\dots,m_1$. Since $\int_{\mathbb{R}^2} \Psi_j \frac{1-|y|^2}{(1+|y|^2)^3}dy=0$ we have that $a_{0j}=0$, $j=1,\dots,m_1$. Now, testing the equation of $\psi_2$ against $P\hat Z_j$, $j=m_1+1,\dots,m$, lead us to deduce that $a_{0j}=0$,$j=m_1+1,\dots,m$. So far, we have shown that $\psi_i \to 0$ in $C_{\hbox{loc}}(S\setminus \{\xi_1^*,\dots,\xi_m^*\})$ and uniformly in $\cup_{j=1}^{m} B_{R \delta_j}(\xi_j)$, for all $R>0$ for both $i=1,2$, in view of $\psi_1-\tilde c_1=\psi_2-\tilde c_2$.

\medskip \noindent Setting $\hat \psi_{i,j} (y)=\psi_i (y_{\xi_j}^{-1}(y))$, $\mathcal{\hat
K}_{j}(y)=[\mathcal{K}_1 +\mathcal{K}_2] (y_{\xi_j}^{-1}(y))$ and $\hat h_j(y)=h(y_{\xi_j}^{-1}(y))$ for $y \in B_{2r_0}(0)$, we have that $
e^{\hat \varphi_j} \Delta \hat \psi_{1,j} + \mathcal{\hat K}_j \hat \psi_{1,j}=\hat h_j + \mathcal{K}_2(y_{\xi_j}^{-1} (y)) [\tilde c_1 - \tilde c_2] $. By now it is rather standard to show that the
operator $\hat L_j=e^{\varphi_j} \Delta + \mathcal{\hat K}_j$ satisfies the maximum principle in $
B_r(0) \sm B_{R\delta_j}(0)$ for $R$ large and $r>0$ small enough, see for example \cite{DeKM}. As a consequence, we get that $\psi_1 \to 0$ in $L^\infty(S)$. Similarly, we also get that $\psi_2 \to 0$ in $L^\infty(S)$. Since $\psi_i=\phi + \tilde c_i(\phi)$ and $\tilde c_i(\phi) \to \tilde c_{i,0}=0$ along a sub-sequence, $\|\psi_i \|_\infty \to 0$ implies $\phi \to 0$ in $L^\infty(S)$, in contradiction with $\|\phi\|_\infty=1$. This completes the proof.
\end{proof}
\noindent We are now ready for
\begin{proof}[{\bf Proof (of Proposition \ref{p2}):}] Since $\|\lap_g PZ_{ij}\|_*\le C$ for all $i=0,1,2$, $j=1,\dots,m$, by Proposition \ref{p1} any solution of \equ{plco} satisfies
$$\|\phi\|_\infty\le
C |\log \de| \bigg[\|h\|_*+ \sum_{i=1}^2\Big(|c_{0i}|+\sum_{j=1}^m |c_{ij}|\Big) \bigg].$$
To estimate the values of the $c_{ij}$'s, test equation \equ{plco} against $PZ_{ij}$, $i=1,2$ and $j=1,\dots,m$:
$$\int_S \phi L(PZ_{ij})dv_g =\int_S h PZ_{ij}dv_g + \sum_{k=1}^2\Big[c_{0k}\sum_{l=0}^m \int_S  \lap_g
PZ_{0l} PZ_{ij}dv_g + \sum_{l=1}^m c_{kl} \int_S \lap_g
PZ_{kl} PZ_{ij}dv_g\Big].$$
Since for $j=1,\dots,m$ we have the following estimates in $C(S)$
\begin{equation}\label{pzij}
PZ_{ij}=\chi_jZ_{ij}+O(\de)\,,\:\:\:i=1,2\,, \qquad  PZ_{0j}=\chi_j(Z_{0j}+2)+O(\de^2|\log\de|),
\end{equation}
it readily follows that $\int_S \lap_g PZ_{kl} PZ_{ij}dv_g=-{32\pi\over 3}\delta_{ki}\delta_{lj}+O(\de)$, where the $\delta_{ij}$'s are the Kronecker's symbols. By Lemma \ref{ewfxi}, (\ref{repla0}), (\ref{intV1eW}), \eqref{intV2etW} and \eqref{pzij} we have that for $i=1,2$
$$L(PZ_{ij})=\chi_j \Delta_g Z_{ij}+e^{U_j} PZ_{ij}+O\Big(\delta^2 +\delta \sum_{k=1}^m e^{U_k}\Big)=
e^{U_j} [PZ_{ij}-e^{-\varphi_j }\chi_j
Z_{ij}]+O\Big(\delta^2+\delta \sum_{k=1}^m e^{U_k}\Big)$$ in view
of $\frac{\int_S V_1 e^W PZ_{ij}dv_g}{\int_S V_1 e^W dv_g}=O(\delta)$ and $\frac{\int_S V_2 e^{-\tau W} PZ_{ij}dv_g}{\int_S V_2 e^{-\tau W} dv_g}=O(\delta)$ for all $j=1,\dots,m$,
leading to $\|L(PZ_{ij})\|_*=O(\delta)$. Similarly, we have that
\begin{eqnarray*}
L(PZ_1) &=&\sum_{j=1}^{m_1} e^{U_j} [PZ_{0j}-\chi_j e^{-\varphi_j } Z_{0j}-
2\chi_j ]+O(\delta^2)+O\bigg(\delta \sum_{k=1}^m e^{U_k}\bigg)
\end{eqnarray*}
in view of $\frac{\int_S V_1 e^{W}  PZ_{0j}dv_g}{\int_S V_1 e^{W}   dv_g}=\frac{2}{m_1}+O(\delta^2|\log \delta|)$ and $\frac{\int_S V_2 e^{-\tau W}  PZ_{0j}dv_g}{\int_S V_2 e^{-\tau W}   dv_g}=O(\delta^2|\log \delta|)$ for $j=1,\dots,m_1$, leading to $\|L(PZ_1)\|_*=O(\delta)$. Also, by using a similar argument for $j=m_1+1,\dots,m$, we find that $\|L(PZ_2)\|_*=O(\delta)$.
Hence, we get that
\begin{equation*}
\begin{split}
\sum_{i=1}^2\Big[|c_{0i}|+ \sum_{j=1}^m
|c_{ij}| \Big] &\leq C' \|h\|_*+\delta |\log
\delta|O\Big(\sum_{i=1}^2 \Big[|c_{0i}|+ \sum_{j=1}^m
|c_{ij}| \Big]  \bigg),
\end{split}
\end{equation*}
yielding to the desired estimates $\|\phi\|_\infty=O(|\log \delta|
\|h\|_*)$ and $\displaystyle \sum_{i=1}^2 \Big[|c_{0i} |+ \sum_{j=1}^m
|c_{ij}|\Big] =O(\|h\|_*)$. To prove the solvability assertion, problem
\eqref{plco} is equivalent to finding $\phi\in H$ such that
\begin{equation*}
\begin{split}
\int_S \langle\nabla \phi, \nabla \psi\rangle_g dv_g=\int_S &\lf[{\la_1 V_1e^W\over \int_S V_1e^W dv_g}\lf(\phi-{\int_S V_1e^W \phi dv_g \over \int_S V_1e^W dv_g}\rg) \rg. \\
& \lf.\; +\, {\la_2\tau^2 V_2e^{-\tau W}\over \int_S V_2e^{-\tau W} dv_g}\lf(\phi-{\int_S V_2e^{-\tau W} \phi dv_g \over \int_S V_2e^{-\tau W} dv_g}\rg) -h\rg]\psi dv_g\qquad \forall \, \psi \in \ml H,
\end{split}
\end{equation*}
where $\ml{H}=\{\phi \in H_0^1(S) \,:\: \int_S \phi \lap_g PZ_{ij} dv_g=\int_S \phi \lap_g PZ_i dv_g=0,\,
i=1,2,\, j=1,\dots,m \}$. With the aid of Riesz representation theorem, the Fredholm's alternative guarantees unique solvability for any $h$ provided that the homogeneous equation has only the trivial solution: for \eqref{plco} with $h=0$, the a-priori estimate (\ref{estmfe1}) gives that $\phi=0$.

\medskip \noindent So far, we have seen that, if $T(h)$ denotes the unique solution $\phi$ of \eqref{plco}, the operator $T$ is a continuous linear map from $\{h \in L^\infty(S):\, \int_S h dv_g =0 \}$, endowed with the $\|\cdot\|_*$-norm, into $\{\phi \in L^\infty(S):\, \int_S \phi dv_g =0 \}$, endowed with $\|\cdot\|_\infty$-norm. The argument below is heuristic but can be made completely rigourous. The operator $T$ and the coefficients $c_{0i},\,c_{ij}$ are differentiable w.r.t. $\xi_{l}$, $l=1,\dots,m$, or $\mu_k$, $	k=1,2$. We shall argue in the same way to obtain (57) in \cite[Appendix A]{EF}, differentiating equation \eqref{plco}, we formally get that $X=\fr_\beta \phi$, where $\beta=\xi_{l}$ with $l=1,\dots,m$ or $\beta=\mu_k$, $k=1,2$, satisfies $L(X)=\ti h(\phi)+\sum_{i}d_{0i}\lap_g PZ+\sum_{i,j} d_{ij}\lap_g PZ_{ij}$, for a suitable choice of $\ti h(\phi)$, $d_{0i}=\fr_\beta c_{0i}$, $d_{ij}=\fr_\beta c_{ij}$, and the
orthogonality conditions  become
$$\int_S X \lap_g PZ_{ij}dv_g =-\int_S \phi  \fr_\beta(\lap_g
PZ_{ij}) dv_g\,, \qquad \int_S X \lap_g PZ_i dv_g=-\int_S \phi \fr_\beta(\lap_g PZ_i)dv_g.$$
Now, finding and estimating suitable coefficients $b_{0i}$, $b_{ij}$ so that $Y=X+\sum_kb_{0k} PZ_k+ \sum_{k,l} b_{kl}PZ_{kl}$ satisfies the orthogonality conditions $\int_S Y \lap_g PZ_i dv_g=\int_S Y \lap_g PZ_{ij}dv_g=0$, the function $X$ can be uniquely expressed as $X=T(f)-\sum_i b_0PZ_i- \sum_{i,j} b_{ij}PZ_{ij}$, where $f=\ti h(\phi)+\sum_ib_{0i} L(PZ_i)+\sum_{i,j} b_{ij}L(PZ_{ij})$. 
Moreover, we find that $\|f\|_* \le C {|\log \de|  \over \de} \|h\|_*,$ for $\beta=\xi_l$ and $\|f\|_*\le C|\log\de| \|h\|_*$ for $\beta=\mu_k$. By (\ref{estmfe1}) we deduce that for any first derivative 
$$\|\fr_{\xi_l} \phi\|_\infty \le C \Big[|\log \de|\|f\|_*+{\|\phi\|_\infty \over \de}\Big] \le C' {|\log \de|^2 \over\de} \|h\|_*.
$$
and $\|\fr_{\mu_k} \phi\|_\infty \le C|\log\de|^2 \|h\|_*$. Differentiating once more in $\mu_j$ the equation satisfied by $\fr_{\mu_i} \phi$ and arguing as above, we finally obtain that
$\|\fr_{\mu_i \mu_j} \phi \|_\infty \le C |\log \de|^3 \|h\|_*$, and the proof is complete.
\end{proof}


\section{Appendix B}\label{appeB}
\noindent We shall argue in the same way to \cite[Proposition 4.2]{EF}, so that by Proposition \ref{p2} we now deduce the following.
\begin{proof}[{\bf Proof (of Proposition \ref{lpnlabis}):}] In terms of the operator $T$, problem \eqref{pnlabis} takes the form $\ml{A}(\phi)=\phi$, where $\ml{A}(\phi):=-T(R+N(\phi))$. After \cite{DeKM,EF,EFP,EGP,F}, a  standard fixed point argument can be used to obtain that $\ml{A}$
is a contraction mapping of $\ml{F}_\nu$ into itself, where
$$\ml{F}_\nu=\left \{\phi\in C(S)\,:\: \|\phi\|_\infty \le \nu
\bigg[\delta |\log \delta| \sum_{j=1}^m |\nabla \log(\rho_j \circ
y_{\xi_j}^{-1})(0)|+\de^{2-\sigma}|\log \de|^2 \bigg] \right \}.$$
Therefore has a unique fixed point $\phi \in \ml{F}_\nu$. 

\medskip \noindent By the Implicit Function Theorem it follows that the map $(\mu,\xi) \to (\phi(\mu,\xi), c_{0i}(\mu,\xi), c_{ij}(\mu,\xi))$ is (at least) twice-differentiable in $\mu$ and one differentiable in $\xi$. Differentiating  $\phi=-T(R+N(\phi))$ w.r.t. $\beta=\xi_l$, $l=1,\dots,m$, or $\beta=\mu$, we get that $\fr_\beta\phi=-\fr_\beta T(R+N(\phi))-
T(\fr_\beta R+\fr_\beta N(\phi))$. By  Lemma \ref{estrr0} and (\ref{estd}) we have that
$$ \|\fr_{\xi_l} T(R+N(\phi))\|_\infty \le C {|\log \de|^2 \over\de} (\|R\|_*+\|N(\phi)\|_*)=O\bigg( |\log \de|^2  \sum_{j=1}^m |\nabla\log (\rho_j \circ y_{\xi_j}^{-1})(0)|+\de^{1-\sigma}|\log \de|^3 \bigg),$$
for $l=1,\dots,m$, in view of $\|\partial_{\xi_l} W\|_\infty \leq \frac{C}{\de}$ and
$$ \|\fr_{\mu} T(R+N(\phi))\|_\infty \le C |\log \de|^2  (\|R\|_*+\|N(\phi)\|_*)=O\bigg( \de|\log \de|^2  \sum_{j=1}^m |\nabla\log (\rho_j \circ y_{\xi_j}^{-1})(0)|+\de^{2-\sigma}|\log \de|^3 \bigg),$$
in view of $\|\partial_\mu W\|_\infty \leq C$. So, differentiating $\fr_\beta N_i(\phi)$ as in \cite[Appendix A]{EF} with $N_i(\phi)$ in \eqref{ni}, we find that
\begin{equation}\label{derivN}
\|\fr_\beta N(\phi) \|_* \le C\lf[\|\fr_\beta W\|_\infty\|\phi\|_\infty^2+\|\phi\|_\infty\|\fr_\beta\phi\|_\infty\rg]
\end{equation} 
and
\begin{eqnarray*}
\|\fr_{\xi_l} N(\phi) \|_* &=& O\bigg(\de |\log \delta|^2 \sum_{j=1}^m |\nabla \log(\rho_j
\circ y_{\xi_j}^{-1})(0)|^2+\de^{3-2\sigma}|\log \de|^4\bigg)
+o\lf(\frac{\|\fr_{\xi_l}\phi\|_\infty}{|\log \de|}\rg)  \\
\|\fr_{\mu} N(\phi) \|_*&=& O\bigg(\de^2 |\log \delta|^2 \sum_{j=1}^m |\nabla \log(\rho_j
\circ y_{\xi_j}^{-1})(0)|^2+\de^{4-2\sigma}|\log \de|^4\bigg)
+o\lf(\frac{\|\fr_\mu\phi\|_\infty}{|\log \de|}\rg).
\end{eqnarray*}
Since $\int_S \chi_j e^{-\varphi_j} e^{U_j}dv_g=\int_{\mathbb{R}^2} \chi(|y|)\frac{8\mu_k^2\de^2 \rho_j(\xi_j)}{(\mu_k^2\de^2 \rho_j(\xi_j)+|y|^2)^2}dy$, either $k=1$ for $j=1,\dots,m_1$ or $k=2$ for $j=m_1+1,\dots,m$, we have that
$$\partial_{\xi_l}\bigg(\int_S \chi_j e^{-\varphi_j} e^{U_j}dv_g\bigg)= 8 \partial_{\xi_l} \log \rho_j(\xi_j) \int_{\mathbb{R}^2}  \frac{1-|y|^2}{(1+|y|^2)^3}+O(\de^2)=O(\de^2)$$
and similarly,
\begin{equation*}
\begin{split}
\partial_{\mu_k} \bigg(\int_S \chi_j e^{-\varphi_j} e^{U_j} dv_g\bigg) & = \int_{\mathbb{R}^2} \chi(|y|)\frac{16 \mu_k\de^2 \rho_j(\xi_j) (|y|^2-\mu_k^2\de^2 \rho_j(\xi_j))}{(\mu_k^2\de^2 \rho_j(\xi_j)+|y|^2)^3} dy =O(\de^2).
\end{split}
\end{equation*}
Since $\varphi_j(\xi_j)=0$ and $\nabla \varphi_j(\xi_j)=0$, we have that $e^{-\varphi_j}=1+O(|y_{\xi_j}(x)|^2)$ and $\partial_\beta(\chi_j e^{-\varphi_j}(x))=O(|y_{\xi_j}(x)|)$, and then $\ds \lab \fr_{\xi_l} W=-\sum_{j=1}^{m_1} \chi_j  e^{U_j}\fr_{\xi_l} U_j + {1\over \tau}\sum_{j=m_1+1}^m\chi_j  e^{U_j}\fr_{\xi_l} U_j+O(\de^{1-\sigma})$, 
in view of $|\partial_{\xi_l} U_j|=O(\frac{1}{\de})$, $l=1,\dots,m$ and  $\ds \lab \fr_\mu W=-\sum_{j=1}^m \chi_j  e^{U_j}\fr_\mu U_j+ {1\over \tau}\sum_{j=m_1+1}^m\chi_j  e^{U_j}\fr_\mu U_j+O(\de^{2-\sigma})$, 
in view of $|\partial_\mu U_j|=O(1)$, where the big
$O$ is estimated in $\|\cdot \|_*$-norm. Note that in $B_{r_0}(\xi_j)$
$$\partial_{\xi_l} W=
\begin{cases}
\partial_{\xi_l} U_j+O(\de^2 |\log \de|+|y_{\xi_j}(x)|+ |\nabla\log(\rho_j \circ y_{\xi_j}^{-1})(0)|),&\text{ for } j\in\{1,\dots,m_1\},\\[0.4cm]
-\dfrac{1}\tau \big[\partial_{\xi_l} U_j+O(\de^2 |\log \de|+|y_{\xi_j}(x)|+ |\nabla\log(\rho_j \circ y_{\xi_j}^{-1})(0)|)\big],&\text{ for } j\in\{m_1+1,\dots,m\},
\end{cases}$$
and
\begin{eqnarray*}
\partial_{\mu_k} W=\begin{cases}
\partial_{\mu_k} U_j-\dfrac{2}{\mu_k}+O(\de^2 |\log \de|),&\text{ for } j\in\{1,\dots,m_1\} \\[0.4cm]
-\dfrac1\tau\Big[\partial_{\mu_k} U_j-\dfrac{2}{\mu_k}+O(\de^2 |\log \de|)\Big],&\text{ for } j\in\{m_1+1,\dots,m\}
\end{cases}.
\end{eqnarray*}
Furthermore, $\partial_{\xi_l} W=O(1)$ and $\partial_{\mu_k} W=O(\de^2|\log\de|)$ in $S\sm \cup_{j=1}^m B_{r_0}(\xi_j)$.  From computations in the proof of Lemma \ref{ewfxi} we find that
\begin{eqnarray} \label{important}
&&\hspace{-1cm}\dfrac{\la_1 V_1 e^W}{\int_S V_1e^W dv_g} =  \frac{\la_1}{8\pi m_1} \sum_{j=1}^{m_1} \chi_j \bigg[1+\Big\langle\frac{\nabla ( \rho_j \circ y_{\xi_j}^{-1})(0)}{ \rho_j(\xi_j)},y_{\xi_j}(x)\Big\rangle +O(|y_{\xi_j}(x)|^2+\de^2 |\log \de|)\bigg] e^{U_j}  \\
&&\hspace{1.4cm} +\ O(\de^2) \chi_{S \setminus \cup_{j=1}^{m_1} B_{r_0}(\xi_j)},\nonumber
\end{eqnarray}
and
\begin{eqnarray} \label{important2}
&&\hspace{-1.5cm}\dfrac{\la_2\tau V_2 e^{-\tau W} }{\int_S V_2 e^{-\tau W} dv_g}  =  \frac{\la_2\tau }{8\pi m_2} \sum_{j=m_1+1}^{m} \chi_j \bigg[1+\Big\langle\frac{\nabla ( \rho_j \circ y_{\xi_j}^{-1})(0)}{ \rho_j (\xi_j)},y_{\xi_j}(x)\Big\rangle +O(|y_{\xi_j}(x)|^2+\de^2 |\log \de|)\bigg] e^{U_j}\\
&&\hspace{1.3cm} +\ O(\de^2) \chi_{S \setminus \cup_{j=m_1+1}^{m} B_{r_0}(\xi_j)}.\nonumber
\end{eqnarray}
By (\ref{2term}), \eqref{Merry1}, \eqref{Merry0}, \eqref{Merry}, (\ref{Merrymu2}), \eqref{important} and \eqref{important2} we deduce for $\fr_\beta R$ 
the estimate
$$\|\partial_{\xi_l} R\|_*+{1\over \de}\|\partial_{\mu_k} R\|_*=O\bigg(\sum_{j=1}^m |\nabla \log(\rho_j \circ y_{\xi_j}^{-1})(0)|+\de^{1-\sigma} |\log \delta|\bigg), \quad l=1,\dots,m,\ k=1,2$$ 
Combining all the estimates, we then get that
$$\|\fr_{\xi_l} \phi\|_\infty=O\bigg(|\log \delta|^2 \sum_{j=1}^m |\nabla \log(\rho_j \circ y_{\xi_j}^{-1})(0)|+\de^{1-\sigma}|\log \de|^3\bigg)
 +o\big(\|\fr_{\xi_l}\phi\|_\infty\big)$$
and
$$\|\fr_{\mu_k}\phi\|_\infty=O\bigg(\de|\log \delta|^2 \sum_{j=1}^m |\nabla \log(\rho_j \circ y_{\xi_j}^{-1})(0)|+\de^{2-\sigma}|\log \de|^3\bigg) +o\big(\|\fr_{\mu_k} \phi\|_\infty\big),$$
which in turn provides the validity of (\ref{cotadphi1bis}). We proceed in the same way to obtain the estimate (\ref{cotadphi1bis}) on $\fr_{\mu_i\mu_j}\phi$, and the proof is complete.
\end{proof}
\noindent Lemma \ref{cpfc0bis} is rather standard and we will omit its proof. Since the problem has been reduced to find c.p.'s of the reduced energy $E_{\lambda_1,\la_2}(\mu, \xi)= J_{\lambda_1,\la_2} (W+\phi(\mu,\xi))$, where $J_{\lambda_1,\la_2}$ is given by \eqref{energy}, the last key step is show that the main asymptotic term of $E_{\lambda_1,\la_2}$ is given by $J_{\lambda_1,\la_2}(W)$.
\begin{proof}[{\bf Proof (of Theorem \ref{fullexpansionenergy}):}] We argue in the same way as in the proof of \cite[Theorem 4.4]{EF}.  For simplicity we write $J$ instead of $J_{\la_1,\la_2}$. Thus, we get that
\begin{eqnarray*}
J(W+\phi)-J(W)
=-{1\over 2}\int_S \lf[R\phi- N(\phi)\phi\rg]dv_g+\int_0^1\!\!\! \int_0^1 [D^2 J(W+ts \phi)-D^2J(W)][\phi,\phi]\,t\, dsdt,
\end{eqnarray*}
so that, it is straighforward to deduce that
$$|J(W+\phi)-J(W)|=O(\|R\|_*\|\phi\|_\infty + \|\phi\|_\infty^3)
=O\left(\delta^2 |\log \delta |\, |\nabla \varphi_m^*(\xi)|^2+
\delta^{3-\sigma}|\log\delta|^2 \right)$$ in view of
(\ref{cotaphi1bis}), $4\pi\grad_{\xi_j}\varphi_m^*(\xi)=\grad\log(\rho_j\circ y_{\xi_j}^{-1})(0)$, for $j=1,\dots,m_1$ and $4\pi\tau^2 \grad_{\xi_j}\varphi_m^*(\xi)=\grad\log(\rho_j\circ y_{\xi_j}^{-1})(0)$, for $j=m_1+1,\dots,m$. Now, differentiating w.r.t. $\be=\xi_{l}$, $l=1,\dots,m$, or $\be=\mu_k$, $k=1,2$ we get that
\begin{eqnarray*}
|\partial_\beta[J(W+\phi)-J(W)]|= O(\|\partial_\beta
R\|_* \|\phi\|_\infty + \|R\|_* \|\partial_\beta \phi\|_\infty+
\|\phi\|_\infty^2  \|\partial_\beta \phi\|_\infty+\|\phi\|_\infty^3 \|\partial_\beta W\|_\infty)
\end{eqnarray*}
by using \eqref{derivN}, so that,
\begin{eqnarray*}
|\partial_{\xi_l}[J(W+\phi)-J(W)]|= O\lf(\big[\delta^2 |\log \delta|\, | \nabla \varphi_m^*(\xi)|^2+\de^{3-\sigma}|\log \de|^2\big]{|\log\de|\over
\de}\rg)
\end{eqnarray*}
and $|\partial_{\mu_k}[J(W+\phi)-J(W)]|=O\lf(\big[\delta^2 |\log \delta|\, | \nabla \varphi_m^*(\xi)|^2+\de^{3-\sigma}|\log \de|^2\big] |\log\de|\rg)$
in view of (\ref{cotaphi1bis})-(\ref{cotadphi1bis}), $\|\partial_{\xi_l} W\|_\infty=O(\frac{1}{\de})$ and $\|\partial_{\mu_k} W\|_\infty=O(1)$. Arguing similarly for the second derivative in $\mu$, we get that $\lf|\partial_{\mu_i \mu_k}[J(W+\phi)-J(W)]\rg|=O\lf(\big[\delta^2 |\log
\delta|\, | \nabla \varphi_m^*(\xi)|^2+\de^{3-\sigma}|\log
\de|^2\big]|\log\de|^2 \rg)$. Combining the previous estimates on the difference $J(W+\phi)-J(W)$ with the expansion of $J(W)=J_{\la_1,\la_2}(W)$ contained in Theorem \ref{expansionenergy}, we deduce the validity of the expansion (\ref{fullJUt}) with an error term which can be estimated (in $C^2(\mathbb{R}^2)$ and $C^1(\Xi)$) like $o(\de^2)+r_{\lambda_1,\la_2}(\mu,\xi)$ as $\de \to 0$, where $r_{\lambda_1,\la_2}(\mu,\xi)$ does satisfy (\ref{rlambda}). \end{proof}



\small
\begin{thebibliography}{99}
\bibitem{AhBaFi} M. Ahmedou, T. Bartsch, T. Fiernkranz, \emph{Equilibria of vortex type Hamiltonians on closed surfaces}, arXiv:2203.13566, 2022.

\bibitem{bp} S. Baraket, F. Pacard, {\em Construction of singular limits for a semilinear elliptic equation in dimension $2$}. Calc. Var. Partial Differential Equations {\bf 6} (1998), no. 1, 1--38.

\bibitem{BaPi}D. Bartolucci, A. Pistoia, {\em Existence and qualitative properties of concentrating solutions for the sinh-Poisson equation}, IMA J. Appl. Math. {\bf 72} (2007), no. 6, 706--729.

\bibitem{BaPiWe}T. Bartsch, A. Pistoia, T. Weth, {\it N -vortex equilibria for ideal fluids in bounded planar domains and new nodal solutions of the sinh-Poisson and the Lane-Emden-Fowler equations}. Comm. Math. Phys. {\bf 297} (2010), no. 3, 653--686.

\bibitem{bjmr} L. Battaglia, A. Jevnikar, A. Malchiodi, D. Ruiz, \emph{A general existence result for the Toda system on compact surfaces}, Adv. Math. {\bf 285} (2015), 937--979

\bibitem{CLMP1} E. Caglioti, P.-L. Lions, C. Marchioro, M. Pulvirenti, \emph{A special class of stationary flows for two-dimensional Euler equations: a statistical mechanics description}. Comm. Math. Phys. {\bf 143} (1992), no. 3, 501--525.

\bibitem{CLMP2} E. Caglioti, P.-L. Lions, C. Marchioro, M. Pulvirenti, \emph{A special class of stationary flows for two-dimensional Euler equations: a statistical mechanics description. II}. Comm. Math.
 Phys. {\bf 174} (1995), no. 2, 229--260.
 
 \bibitem{ChI} D. Chae, O. Imanuvilov, \emph{The existence of non-topological multivortex solutions in relativistic self-dual Chern-Simons theory}, Comm. Math. Phys. {\bf 215} (2000) 119--142.

\bibitem{ChChL} S-Y.A. Chang, C.C. Chen, C.-S. Lin, \emph{Extremal functions for a mean field equation in two dimension}.  Lectures on partial differential equations,  61--93, New Stud. Adv. Math., 2, Int. Press, Somerville, MA, 2003.

\bibitem{ChGY} S.-Y. A. Chang, M. J. Gursky and P. C. Yang, \emph{The scalar curvature equation on 2- and 3-spheres}. Calc. Var. Partial Differential Equations {\bf 1} (1993), no. 2, 205--229.

\bibitem{ChY} S.Y.A. Chang and P. C. Yang, \emph{Prescribing Gaussian curvature on $S^2$}. Acta Math. {\bf 159} (1987), no. 3-4, 215--259.

\bibitem{ChKi} S. Chanillo, M. Kiessling, \emph{Rotational symmetry of solutions of some nonlinear problems in statistical mechanics and in geometry}. Comm. Math. Phys. {\bf 160} (1994), no. 2, 217--238.

\bibitem{CL0} C.C. Chen and C.S. Lin, \emph{Sharp estimates for solutions of multi-bubbles in compact Riemann surfaces}. Comm. Pure Appl. Math. {\bf 55} (2002), no. 6, 728--771.

\bibitem{CL} C.C.Chen and C.S.Lin, \emph{Topological degree for a mean field equation on Riemann surfaces}, Comm. Pure Appl. Math. {\bf 56} (2003) 1667--1727.

\bibitem{CLW} C.C. Chen, C.-S. Lin, G. Wang, {\it Concentration phenomena of two-vortex solutions in a Chern-Simons model}, Ann. Sc. Norm. Sup. Pisa Cl. Sci. (5) {\bf 3} (2004), 367--397.

\bibitem{Chern} S.-S. Chern, \emph{An elementary proof of the existence of isothermal parameters on a surface}. Proc. Amer. Math. Soc. {\bf 6} (1955), 771--782.

\bibitem{DaE} T. D'Aprile, P. Esposito, \emph{Equilibria of point-vortices on closed surfaces}, Ann. Sc. Norm. Super. Pisa Cl. Sci. {\bf 17} (2017), no. 1, 287--321


\bibitem{Dja} Z. Djadli, \emph{Existence result for the mean field problem on Riemann surfaces of all genuses}, Commun. Contemp. Math. {\bf 10} (2008), 205--220.

\bibitem{DEFM} M. del Pino, P. Esposito, P. Figueroa, M. Musso, \emph{Non-topological condensates for the self-dual Chern-Simons-Higgs model}, Comm. Pure Appl. Math. {\bf 68} (2015), 1191--1283.

\bibitem{DeKM} M. del Pino, M. Kowalczyk, M. Musso, \emph{Singular limits in Liouville-type equations}, Cal. Var. P.D.E., {\bf 24} (2005), 47--81.

\bibitem{dmr} M. del Pino, M. Musso, B. Ruf, \emph{New solutions for Trudinger–Moser critical equations in $\mathbb{R}^2$}, J. Funct. Anal. {\bf 258} (2010) 421--457.



\bibitem{DJLW} W. Ding, J. Jost, J. Li, G. Wang, \emph{Existence results for mean field equations}. Ann. Inst. H. Poincaré Anal. Non Linéaire {\bf 16} (1999), no. 5, 653--666.


\bibitem{EF} P. Esposito, P. Figueroa, \emph{Singular mean field equations on compact Riemann surfaces}, Nonlinear Analysis {\bf 111} (2014), 33--65.

\bibitem{EFP} P. Esposito, P. Figueroa, A. Pistoia \emph{On the mean field equation   with  variable intensities on pierced domains}, Nonlinear Analysis {\bf 190} (2020) 111597 

\bibitem{EGP} P. Esposito, M. Grossi,A. Pistoia, \emph{On the existence of blowing-up solutions for a mean field equation}. Ann. Inst. H. Poincaré Anal. Non Linéaire {\bf 22} (2005), no. 2, 227--257.

\bibitem{EMP} P. Esposito, M. Musso, A. Pistoia, \emph{Concentrating solutions for a planar elliptic problem involving nonlinearities with large exponent}. J. Diff. Equ. {\bf 227}(1), 29--68 (2006) 7.


\bibitem{EW} P. Esposito, J. Wei, \emph{Non-simple blow-up solutions for the Neumann two-dimensional sinh-Gordon equation}. Calc. Var. Partial Differential Equations {\bf 34} (2009), no. 3,
341--375.

\bibitem{F} P. Figueroa, \emph{Singular limits for Liouville-type equations on the flat torus}, Calc. Var. Partial Differential Equations {\bf 49} (1--2) (2014), 613--647.

\bibitem{F2} P. Figueroa, \emph{A note on sinh-Poisson equation with variable intensities on pierced domains},  Asymptotic Analysis, {\bf 122} (2021) 327--348. 

\bibitem{F3} P. Figueroa, \emph{Sign-changing bubble tower solutions for sinh-Poisson type equations on pierced domains}, in preparation.

\bibitem{FM} P. Figueroa, M. Musso, \emph{Bubbling solutions for Moser-Trudinger type equations on compact Riemann surfaces},  J. Funct. Anal. {\bf 275}  (2018), no.10, 2684--2739.

\bibitem{FIT} P. Figueroa, L. Iturriaga, E. Topp, \emph{Concentrating solutions for a sinh-Poisson equation with Robin boundary condition}, in preparation.

\bibitem{GP} M. Grossi, A. Pistoia, \emph{Multiple Blow-Up Phenomena for the Sinh-Poisson Equation}, Arch. Rational Mech. Anal. {\bf 209} (2013) 287--320.

\bibitem{J0} A. Jevnikar, {\it An existence result for the mean-field equation on compact surfaces in a doubly supercritical regime}, Proc. Royal Soc. Edinburgh A {\bf 143} (2013), 1021--1045.

\bibitem{J1} A. Jevnikar, {\it New existence results for the mean field equation on compact surfaces via degree theory},  Rend. Semin. Mat. Univ. Padova {\bf 136} (2016), 11--17.

\bibitem{j2} A. Jevnikar, \emph{Blow-up analysis and existence results in the supercritical case for an asymmetric mean field equation with variable intensities}, J. Diff. Eq. {\bf 263} (2017), no. 2, 972--1008

\bibitem{JWY} A. Jevnikar, J. Wei, W. Yang, {\it On the Topological degree of the mean field equation with two parameters}, Indiana Univ. Math. J. {\bf 67} (2018), no. 1, 29--88.

\bibitem{JWY2} A. Jevnikar, J. Wei, W. Yang, {\it Classification of blow-up limits for Sinh-Gordon equation}, Differential and Integral Equations {\bf 31} (2018), no. 9/10, 657--684.

\bibitem{JWYZ} J. Jost, G. Wang, D. Ye, C. Zhou, {\it The blow up analysis of solutions to the elliptic sinh-Gordon equation}. Calc. Var. Partial Diff. Equ. {\bf 31}(2), 263-276 (2008).

\bibitem{KW} J. L. Kazdan and F. W. Warner, \emph{Curvature functions for compact 2-manifolds}. Ann. of Math. (2) {\bf 99} (1974), 14--47.

\bibitem{K} M. K.-H. Kiessling, \emph{Statistical mechanics of classical particles with logarithmic interactions}. Comm. Pure Appl. Math. {\bf 46} (1993), no. 1, 27--56.

\bibitem{Li0} Y.Y. Li, \emph{On a singularly perturbed elliptic equation}. Adv. Differential Equations {\bf 2}
(1997), 955--980.

\bibitem{Li} Y. Y. Li, \emph{Harnack type inequality: the method of moving planes}. Comm. Math. Phys. {\bf 200} (1999), no. 2, 421--444.

\bibitem{LSh} Y.-Y. Li, I. Shafrir, {\em Blow-up analysis for solutions of $-\Delta u=Ve\sp u$ in dimension two}. Indiana Univ. Math. J. {\bf 43} (1994), no. 4, 1255--1270.

\bibitem{LinYan} C.-S. Lin, S. Yan, {\it Existence of Bubbling Solutions for Chern-Simons Model on a Torus}. Arch. Ration. Mech. Anal. {\bf 207} (2013), no. 2, 353-392.

\bibitem{Mal} A. Malchiodi, \emph{Morse theory and a scalar field equation on compact surfaces}, Adv. Differential Equations {\bf 13} (2008), 1109--1129.

\bibitem{NT} M. Nolasco and G. Tarantello, \emph{Double Vortex condensates in the Chern-Simons-Higgs theory}, Cal. Var. P.D.E., {\bf 9}, 31--94 (1999).

\bibitem{OhSu} H. Ohtsuka, T. Suzuki, {\it Mean field equation for the equilibrium turbulence and a related functional inequality}, Adv. Differential Equations {\bf 11} (2006) 281--304.

\bibitem{o} L. Onsager, \emph{Statistical hydrodynamics}. Nuovo Cimento (9) {\bf 6} (1949), 279--287.

\bibitem{pr1} A. Pistoia, T. Ricciardi, \emph{Concentrating solutions for a Liouville type equation with variable intensities in 2D-turbulence}, Nonlinearity {\bf 29} (2016), no. 2, 271--297.

\bibitem{pr2} A. Pistoia, T. Ricciardi, \emph{Sign-changing tower of bubbles for a sinh-Poisson equation with asymmetric exponents}, Discrete Contin. Dyn. Syst. {\bf 37} (2017), 5651--5692.

\bibitem{R} T. Ricciardi, {\it Mountain-pass solutions for a mean field equation from two-dimensional turbulence}. Differential Integral Equations {\bf 20} (2007), no. 5, 561--575.

\bibitem{rt} T. Ricciardi, R. Takahashi, \emph{Blow-up behavior for a degenerate elliptic sinh-Poisson equation with variable intensities}, Calc. Var. Partial Differential Equations {\bf 55} (2016), Paper No. 152, 25 pp.

\bibitem{rtzz}T. Ricciardi, R. Takahashi, G. Zecca, X. Zhang, \emph{On the existence and blow-up of solutions for a mean field equation with variable intensities}, Atti Accad. Naz. Lincei Rend. Lincei Mat. Appl. {\bf 27} (2016), 413--429.

\bibitem{rz} T. Ricciardi, G. Zecca, \emph{Minimal blow-up masses and existence of solutions for an asymmetric sinh- Poisson equation}, Math. Nachr. {\bf 290} (2017), no. 14-15, 2375--2387

\bibitem{ss} K. Sawada, T. Suzuki, \emph{Derivation of the equilibrium mean field equations of point vortex and vortex filament system}, Theoret. Appl. Mech. Japan {\bf 56} (2008), 285--290.

\bibitem{T} G. Tarantello, \emph{Multiple condensate for Chern-Simons-Higgs theory}, J. Math. Phys., {\bf 37} (1996), no. 8, 3769--3796.

\bibitem{Tbook} G. Tarantello, {\em Self-dual gauge field theories. An analytical approach}, Progress in Nonlinear Differential Equations and their Applications, 72. Birkh\"auser, Boston, 2008.

\bibitem{Zhou0} C. Q. Zhou, {\it Existence of solution for mean field equation for the equilibrium turbulence}. Nonlinear Anal. {\bf 69} (2008), no. 8, 2541--2552.

\bibitem{Zhou} C. Q. Zhou, {\it Existence result for mean field equation of the equilibrium turbulence in the super critical case}. Commun. Contemp. Math. {\bf 13} (2011), no. 4, 659--673


\end{thebibliography}
\end{document}